\documentclass[12pt,a4paper]{amsart}
\usepackage{amsmath}
\usepackage{amssymb}
\usepackage{constants}

\makeatletter

\theoremstyle{plain}
\newtheorem{theorem}
{\indent Theorem}[section]
\newtheorem{proposition}[theorem]
{\indent Proposition}
\newtheorem{lemma}[theorem]
{\indent Lemma}
\newtheorem{corollary}[theorem]
{\indent Corollary}

\theoremstyle{definition}

\newtheorem{remark}[theorem]
{\indent Remark}

\theoremstyle{plain}

\numberwithin{equation}{section}

\newconstantfamily{c}{symbol=c}
\newconstantfamily{K}{symbol=K}

\newcommand\al{\alpha}
\newcommand\Cls[2]{C^#1_x\cap C^#2_t}
\newcommand\dds{\frac{d}{ds}}
\newcommand\ddt{\frac{d}{dt}}
\newcommand\dl{\delta}
\newcommand\DD{\nabla}
\newcommand\Ds{\DD_{\!s}}
\newcommand\Dt{\DD_{\!t}}
\newcommand\Dx{\DD_{\!\!x}}
\newcommand\ep{\varepsilon}
\newcommand\gm{\gamma}
\newcommand\Gm{\Gamma}
\newcommand\wh{\widehat}
\newcommand\Ip[1]{\langle#1\rangle}
\newcommand\Ker{\operatorname{Ker}}

\newcommand\Lm{\Lambda}
\newcommand\nr[1]{\mathopen|#1\mathclose|}
\newcommand\Nr[1]{\mathopen\|#1\mathclose\|}
\newcommand\ol{\overline}
\newcommand\pd{\partial}
\newcommand\td{\widetilde}
\newcommand\te{\theta}

\allowdisplaybreaks

\begin{document}


\title[\sc Motion of elastic wire]
{%
Motion of an elastic wire with thickness\\
in a Riemannian manifold
}

\author[\sc N.~Koiso]
{Norihito Koiso}



\thanks{%
Supported by the Grant-in-Aid for the Scientific Research (C) Grant Number 26400069.
}

\address{%
Norihito KOISO\\
Department of Mathematics\\
Faculty of Science\\
Osaka University\\
Toyonaka, Osaka, 560-0043\\
JAPAN
}

\begin{abstract}
There are several types of equation of motion of elastic wires.
In this paper, we treat an equation taking account of the thickness of wire.
The equation was introduced by Caflisch and Maddocks on plane curves, and they proved the existence of solutions.
Koiso and Sugimoto generalized the result to any dimensional Euclidean space.
In this paper, we will prove the existence of solutions on any riemannian manifold.
\end{abstract}

\maketitle

\section
{Introduction and results}

An {\em elastic wire} is a closed arcwise parametrized curve $\gm:S^1\to{\bf R}^N$ endowed with {\em elastic energy}
\begin{equation}
U(\gm):=\Nr{\gm_{xx}}^2:=\int\nr{\gm_{xx}}^2\,dx.
\end{equation}
According to Hamilton's principle, an equation of motion of an elastic wire is determined by the potential energy $U(\gm)$ and other kinetic energy.
For a start, we may consider $F(\gm)=\Nr{\gm_t}^2$ as kinetic energy.
The equation of motion we get is
\begin{equation}
\gm_{tt}+\gm_{xxxx}+(\mu\gm_x)_x=0,\quad
\nr{\gm_x}^2=1.
\end{equation}
Here, $\mu=\mu(x,t)$ is an unknown function determined by the condition $\nr{\gm_x}^2=1$.
The principal part of the equation is the $1$-dimensional plate equation.

For this equation, R.~Courant and D.~Hilbert \cite{CH} introduced its linear version.
Recently, there are researches by the author \cite{K-PlE} (case of Euclidean space ${\bf R}^n$),
A.~Burchard and L.~E.~Thomas \cite{BT} (case of ${\bf R}^3$, uses Hasimoto transformation on the vortex filament equation),
\cite{K-PlM} (case of Riemann manifold $(M,g)$).

We may consider a parabolic equation by replacing Newton dynamics by gradient flow method.
\begin{equation}
\gm_{t}+\gm_{xxxx}+(\mu\gm_x)_x=0,\quad
\nr{\gm_x}^2=1.
\end{equation}
For this direction, there are researches by the author
\cite{K-PaE} (case of Euclidean space ${\bf R}^n$),
\cite{K-PaM} (case of Riemann manifold $(M,g)$).
There is another equation which preserves curve length but does not preserve arcwise parametrization: Y.~Wen \cite{W} (case of Euclidean space ${\bf R}^n$).

In this paper, we consider an elastic wire with thickness under Newton dynamics.
When we change direction of infinitesimal part of center line of an elastic wire with thickness, its surrounding material really moves.
Therefore, we add quantity $\Nr{\gm_{xt}}^2$ to the previous kinetic energy $\Nr{\gm_{t}}^2$, and we apply Hamilton's principle to new kinetic energy
$F(\gm)=\Nr{\gm_t}^2+\Nr{\gm_{xt}}^2$.
We get the equation of motion:
\begin{equation}
\label{eq:wave-eq-of-elastic-wire}
\gm_{tt}-\gm_{xttx}+\gm_{xxxx}+(\mu\gm_x)_x=0,\quad
\nr{\gm_x}^2=1.
\end{equation}
The principal part of the equation is a hyperbolic equation for $\gm_{xx}$.
This equation was introduced by R.~Caflish \& J.~Maddocks (1984) \cite{CM} in the case of Euclidean plane ${\bf R}^2$.
They proved the existence and the uniqueness of the solution to the equation.
To prove it, they converted the equation to a hyperbolic equation for $\theta$, where $\gm_x=:\xi=(\cos\theta,\sin\theta)$.
Hence, we cannot directly generalize the result to higher dimensional case.
The existence and uniqueness of the solution in the case of general dimensional case was given by N.~Koiso \& M.~Sugimoto (2010) \cite{KS}.

Because this equation is physically natural and is given by Hamilton's principle with a clear functional,
it seems to be reasonable that the equation can be generalized to Riemannian manifolds,
i.~e., the Riemannian version of the equation can be solved.
This perturbation of equation is mild, because the principal part is preserved.
At first sight the perturbation seems to be easy because hyperbolic equations are stable under modifying lower derivatives.
However, direct generalization is not possible, because the proof on Euclidean spaces ${\bf R}^N$ in \cite{KS} depends on converting the equation to an equation for $\xi:=\gm_x\in S^{N-1}$.

We consider equation \eqref{eq:wave-eq-of-elastic-wire} in a Riemannian manifold $(M,g)$, hereinafter.
From now on, curves $\gm$ are closed curves$:S^1={\bf R}/{\bf Z}\to M$.
By rescaling the metric $g$ of the Riemannian manifold $M$,
we may assume that the initial curve is of length $1$ and arcwise parametrized.
We denote by $\Ip{*,*}$ the $L_2$ inner product for $x$ direction,
and by $\Nr*$ the $L_2$ norm.
I.~e., $\Nr{*}^2=\Ip{*,*}$.
We also use $L_2$ norm for pointwise norm.
That is, the norms of $v=(v^i)$, $a=(a^j{}_i)$ and $b=(b_i{}^k{}_j)$ are defined by
$\nr v^2=\sum_i(v^i)^2$,
$\nr a^2=\sum_{i,j}(a^j{}_i)^2$
and $\nr b^2=\sum_{i,j,k}(b_i{}^k{}_j)^2$,
respectively, in orthonormal systems.

For potential energy $U(\gm)=\Nr{\Dx\gm_x}^2$ and kinetic energy $F(\gm)=\Nr{\Dt\gm_x}^2+\Nr{\gm_t}^2$,
the first variation of the action integral becomes
\begin{equation}
\begin{aligned}
\frac12\frac{d}{ds}\int_0^TF(\gm)-U(\gm)\,dt
&=\frac12\frac{d}{ds}\int_0^T\Nr{\Dt\gm_x}^2+\Nr{\gm_t}^2-\Nr{\Dx\gm_x}^2\,dt\\
&=\int_0^T\Ip{-\Dt\gm_t+\Dx\Dt^2\gm_x-\Dx^3\gm_x+\Psi,\gm_s}\,dt,\\
\end{aligned}
\end{equation}
where $\Psi:=R(\gm_x,\Dx\gm_x)\gm_x-R(\gm_x,\Dt\gm_x)\gm_t$.

From this, we will see that the equation of motion becomes
\begin{equation}
\label{eq:base-single}
-\Dt\gm_t+\Dx\Dt^2\gm_x-\Dx^3\gm_x+\Psi=\Dx(\mu\gm_x),
\quad \nr{\gm_x}^2=1.
\end{equation}
Here, $\mu=\mu(x,t)$ is the Lagrange's multiplier.
The constrained condition for the initial data $\{\gm(x,0),\gm_t(x,0)\}$ is $\nr{\gm_x(x,0)}^2=1$ and $g(\Dx\gm_t(x,0),\gm_x(x,0))=0$.
We call this equation {\em the wave equation of motion of elastic wire}, or simply the equation of motion in this paper.

\begin{theorem}[Theorem \ref{theorem:long-time-C3-existence}]
The equation of motion \eqref{eq:base-single} has a unique short-time solution for any initial data $\{\gm(x,0),\gm_t(x,0)\}$ which is not a geodesic and satisfies the constrained condition.
The solution extends to an infinite time solution, if it stays away from geodesics.
\end{theorem}

Here, the condition ``the initial data is not a geodesic'' and ``the solution stays away from geodesics'' are natural on our setting.
For example, if the initial curve $\gm(x,0)$ is an isolated shortest geodesic,
any vector field $\eta(x)$ orthogonal to $\gm_x(x,0)$ satisfies the constrained condition as initial velocity,
but only translation of parameter $x$ is allowed as the solution.
And, because geodesics are singular points of the space of closed curves of given length,
it is difficult to analyze behaviour of curves around geodesics.

This paper is constructed as follows.
First, we derive the equation of motion \eqref{eq:base-single} in section \ref{section:derivation-equation},
and decompose it to a coupled system \eqref{eq:base-coupled} of equations in vector fields $\xi=\gm_x$, $\eta=\gm_t$ and $\te$.
\begin{equation*}
\left\{
\begin{aligned}
{\rm(O_\te)}\quad
&-D_x^2\te+\te^\perp
=D_x\Psi+\Phi,\\
{\rm(W_\xi)}\quad
&(D_t^2\xi-D_x^2\xi)^\perp=\te^\perp,
\quad \nr{\xi}^2=1,\\
{\rm(O_\eta)}\quad
&D_t\eta=D_x\te+\Psi+D_x\xi,\\
{\rm(O_\gm)}\quad
&\gm_t=\eta.
\end{aligned}
\right.
\end{equation*}
Here, $D_x$ and $D_t$ are differential operators approximating $\Dx$ and $\Dt$ respectively,
$\Psi$ and $\Phi$ are expressed by $\xi$ and $\eta$,
and $*^\perp$ is the orthogonal part to $\xi$.
We call this system {\em the coupled system of equations of motion of elastic wire}, or simply the coupled system in this paper.

Next, we prove existence of short-time solutions to the coupled system \eqref{eq:base-coupled} in section \ref{section:short-time-exisistence}.
For that, we need to analyze each decomposed equation.
This part is most important in this paper.
We will solve these equations in function spaces $\{\gm\mid\gm_t\in C^0\}$, $\{\eta\mid\eta_t\in C^0\}$, $\{\xi\in C^1\}$ and $\{\te\mid\te_x\in C^0\}$.

In subsection \ref{subsection:O-eta+O-gm}, we analyze equations $\rm(O_\eta)$ and $\rm(O_\gm)$.
They are simple ordinal differential equations, but their estimation gives idea of estimation of equation $\rm(W_\xi)$.
In subsection \ref{subsection:W-xi}, we analyze equation $\rm(W_\xi)$.
Equation $\rm(W_\xi)$ is converted to a standard semilinear wave integral equation.
However, since we don't assume that $\gm_x$ is continuous, we need careful treatment of it.
In subsection \ref{subsection:Bgmxi}, we define {\em curve bentness} $B(\gm)$ and show some basic properties.
The curve bentness guarantees that $\gm$ is not a geodesic in $H^1$ topology.
In subsection \ref{subsection:O-te}, we analyze equation $\rm(O_\te)$ using $B(\gm)$.
This part reflects geometry of the Riemannian metric $g$.
In subsection \ref{subsection:short-time}, we prove the short-time existence of solutions to the coupled system \eqref{eq:base-coupled} (Proposition \ref{proposition:short-time-C2-existence}).
The solution will give a solution $\gm\in C^3$ to the equation of motion \eqref{eq:base-single} (Proposition \ref{proposition:short-time-C3-exsitence}).

In conclusion, we prove the long-time existence of solutions (Theorem \ref{theorem:long-time-C3-existence}) in section \ref{section:long-time-existence}.
We also prove that the solution is of class $C^\infty$ if the initial data is of class $C^\infty(x)$ (Theorem \ref{theorem:long-time-Cinfty-existence}).

\section
{Derivation and decomposition of equation}
\label{section:derivation-equation}

We describe derivation of equation in details.
Let $\gm=\gm(x,t,s)$ be a variation of a motion $\gm(x,t)$ of a curve.
The variational vector field is $\gm_s$.
The first variation of the Hamilton's functional becomes as follows.
\begin{equation}
\begin{aligned}
\frac12\int_0^T(F(\gm)-U(\gm))\,dt
&=\frac12\frac{d}{ds}\int_0^T\Nr{\Dt\gm_x}^2+\Nr{\gm_t}^2-\Nr{\Dx\gm_x}^2\,dt\\
&=\int_0^T\Ip{\Dt\gm_x,\Ds\Dt\gm_x}+\Ip{\gm_t,\Ds\gm_t}
	-\Ip{\Dx\gm_x,\Ds\Dx\gm_x}\,dt.
\end{aligned}
\end{equation}
Since each term is,
\begin{align}
&\begin{aligned}
&\int_0^T\Ip{\Dt\gm_x,\Ds\Dt\gm_x}\,dt
=\int_0^T\Ip{\Dt\gm_x,\Dt\Ds\gm_x+R(\gm_s,\gm_t)\gm_x}\,dt\\
&\qquad=\int_0^T-\Ip{\Dt^2\gm_x,\Dx\gm_s}+\Ip{R(\gm_x,\Dt\gm_x)\gm_t,\gm_s}\,dt\\
&\qquad=\int_0^T\Ip{\Dx\Dt^2\gm_x,\gm_s}-\Ip{R(\gm_x,\Dt\gm_x)\gm_t,\gm_s}\,dt,
\end{aligned}\\
&\int_0^T\Ip{\gm_t,\Ds\gm_t}\,dt
=\int_0^T\Ip{\gm_t,\Dt\gm_s}\,dt
=\int_0^T-\Ip{\Dt\gm_t,\gm_s}\,dt,\\
&\begin{aligned}
&-\Ip{\Dx\gm_x,\Ds\Dx\gm_x}
=-\Ip{\Dx\gm_x,\Dx\Ds\gm_x+R(\gm_s,\gm_x)\gm_x}\\
&\qquad=\Ip{\Dx^2\gm_x,\Dx\gm_s}+\Ip{R(\gm_x,\Dx\gm_x)\gm_x,\gm_s}
=-\Ip{\Dx^3\gm_x,\gm_s}+\Ip{R(\gm_x,\Dx\gm_x)\gm_x,\gm_s}.
\end{aligned}
\end{align}
Summing up them, we see that the first variation is
\begin{equation}
2\int_0^T\Ip{-\Dt\gm_t+\Dx\Dt^2\gm_x-\Dx^3\gm_x+\Psi,\gm_s}\,dt.
\end{equation}

By Hamilton's principle,
the first variation have to vanish for any variational vector field $\gm_s$ satisfying
$0=\pd_s(\nr{\gm_x}^2)=2g(\gm_x,\Ds\gm_x)=2g(\gm_x,\Dx\gm_s)$.
It is equivalent to that,
the first variation vanishes if
\begin{equation}
0=-\int_0^T\int\mu g(\gm_x,\Dx\gm_s)\,dxdt=\int_0^T\Ip{\Dx(\mu\gm_x),\gm_s}\,dt
\end{equation}
for all functions $\mu=\mu(x,t)$.
In other words,
using $q:=-\Dt\gm_t+\Dx\Dt^2\gm_x-\Dx^3\gm_x+\Psi$ and $X:=$ the space $\{\Dx(\mu\gm_x)\}$,
it holds that $\gm_s\perp q$ if $\gm_s\perp X$,
with respect to the $L_2$-inner product on $S^1\times[0,T]$.
Therefore,
$q$ belongs to $X$.
That is, the equation \eqref{eq:base-single} holds.

We cannot directly obtain a priori estimate of $\mu$ in the form \eqref{eq:base-single}.
Instead, we will decompose \eqref{eq:base-single} and eliminate $\mu$.
Letting $\xi:=\gm_x$, we get
\begin{equation}
\begin{aligned}
&-\Dt\gm_t+\Dx\Dt^2\xi-\Dx^3\xi+\Psi=\Dx(\mu\xi),\\
&\Dx(\Dt^2\xi-\Dx^2\xi-\mu\xi)=\Dt\gm_t-\Psi.
\end{aligned}
\end{equation}
We rewrite it to $\Dx\al=\Dt\gm_t-\Psi$ using $\al:=\Dt^2\xi-\Dx^2\xi-\mu\xi$.
To express the $\gm_t$ by $\xi$, we differentiate both sides with respect to $x$.
\begin{equation}
\begin{aligned}
\Dx^2\al&=\Dx\Dt\gm_t-\Dx\Psi
=\Dt\Dx\gm_t+R(\gm_x,\gm_t)\gm_t-\Dx\Psi\\
&=\Dt^2\gm_x+R(\gm_x,\gm_t)\gm_t-\Dx\Psi
=\Dt^2\xi+R(\gm_x,\gm_t)\gm_t-\Dx\Psi.
\end{aligned}
\end{equation}

The definition of $\al$ implies that $(\Dt^2\xi-\Dx^2\xi)^\perp=\al^\perp$.
On the other hand, from $\nr{\xi}^2\equiv1$, we have
$g(\xi,\Dt\xi)=0$,
$g(\xi,\Dt^2\xi)=\pd_t(g(\xi,\Dt\xi))\allowbreak-\nr{\Dt\xi}^2=-\nr{\Dt\xi}^2$.
Similarly we have $g(\xi,\Dx\xi)=0$ and $g(\xi,\Dx^2\xi)=-\nr{\Dx\xi}^2$.
Hence, $\xi$-component of $\Dt^2\xi-\Dx^2\xi$ becomes
\begin{equation}
(\nr{\Dx\xi}^2-\nr{\Dt\xi}^2)\xi
=(\Dt^2\xi-\Dx^2\xi)-(\Dt^2\xi-\Dx^2\xi)^\perp
=\Dx^2(\al-\xi)-R(\gm_x,\gm_t)\gm_t+\Dx\Psi-\al^\perp.
\end{equation}
Therefore, by letting $\te:=\al-\xi$, we have
\begin{equation}
\Dx^2\te
=\te^\perp-(\nr{\Dt\xi}^2-\nr{\Dx\xi}^2)\xi+R(\xi,\gm_t)\gm_t-\Dx\Psi.
\end{equation}

Noting that $\Dt\gm_t=\Dx\al+\Psi=\Dx\te+\Dx\xi+\Psi$,
and introducing
$\eta:=\gm_t$,
we collect equations for
$\te$, $\xi$, $\eta$ and $\gm$
as follows.
\begin{equation}
\label{eq:base-coupled-pre}
\begin{aligned}
&\left\{
\begin{aligned}
&-\Dx^2\te+\te^\perp=\Dx\Psi+\Phi,\\
&(\Dt^2\xi-\Dx^2\xi)^\perp=\te^\perp,
\quad\nr\xi^2=1,\\
&\Dt\eta=\Dx\te+\Psi+\Dx\xi,\\
&\gm_t=\eta,\\
\end{aligned}
\right.\\
&\qquad\quad\Psi=R(\xi,\Dx\xi)\xi-R(\xi,\Dt\xi)\eta,
\quad\Phi=(\nr{\Dt\xi}^2-\nr{\Dx\xi}^2)\xi-R(\xi,\eta)\eta,\\
&\qquad\quad*^\perp:=*-g(*,\xi)\xi.
\end{aligned}
\end{equation}

Next, we give the coordinate expression of this system.
Let $\{e_i\}$ be an orthonormal frame field around the image of the initial curve $\gm(x)$.
We define a family of smooth functions $\{\Gm_i{}^k{}_j\}$ by
$\DD_{e_i}e_j=\Gm_i{}^k{}_je_k$.
It is modified Christoffel symbols.
When the conversion law of $\pd/\pd x^i$ and $e_i$
is given as $\pd/\pd x^i=h^j{}_ie_j$, we have
$\Dx p=p^i_xe_i+\Gm_i{}^k{}_j\gm^\ell_xh^i{}_\ell p^je_k$ for $p=p^ie_i$.
We define a vector valued bilinear form $\Gm$ by
$\Gm(e_i,e_j)=\Gm_i{}^k{}_je_k$.
We introduce differential operators
\begin{equation}
\begin{aligned}
D_xp&:=\Dx p-\Gm(\gm_x-\xi,p)
=p_x+\Gm(\xi,p),\\
D_tp&:=\Dt p-\Gm(\gm_t-\eta,p)
=p_t+\Gm(\eta,p),
\end{aligned}
\end{equation}
and use them instead of $\Dx$ and $\Dt$.
We use them only for proof of short time existence of solutions,
because the differential operators $D_x$ and $D_t$ depend on the choice of an orthonormal frame $\{e_i\}$.
We have
$D_xe_i=\DD_\xi e_i=\Gm(\xi,e_i)$
and
$D_te_i=\DD_\eta e_i=\Gm(\eta,e_i)$.
It holds that $D_x=\Dx$ if $\xi=\gm_x$,
and $D_t=\Dt$ if $\eta=\gm_t$.
Also, since
\begin{equation}
0=e_k(g(e_i,e_j))
=g(\DD_{e_k}e_i,e_j)+g(e_i,\DD_{e_k}e_j)
=\Gm_k{}^j{}_i+\Gm_k{}^i{}_j,
\end{equation}
$g(\Gm(v,p),q)$ is anti-symmetric with respect to $p$ and $q$.
Therefore,
\begin{equation}
\begin{aligned}
\pd_x(g(p,q))&=g(p_x,q)+g(p,q_x)
=g(p_x+\Gm(\xi,p),q)+g(p,q_x+\Gm(\xi,q))\\
&=g(D_xp,q)+g(p,D_xq).
\end{aligned}
\end{equation}
From this, the formula of integration by parts holds, i.e.,
\begin{equation}
\Ip{D_xp,q}+\Ip{p,D_xq}=0.
\end{equation}

We replace $\Dt$ and $\Dx$ in \eqref{eq:base-coupled-pre} by $D_t$ and $D_x$,
and call it the coupled system of equations.
\begin{equation}
\label{eq:base-coupled}
\begin{aligned}
&\left\{
\begin{aligned}
{\rm(O_\te)}\quad
&-D_x^2\te+\te^\perp
=D_x\Psi+\Phi,\\
{\rm(W_\xi)}\quad
&(D_t^2\xi-D_x^2\xi)^\perp=\te^\perp,
\quad\nr\xi^2=1,\\
{\rm(O_\eta)}\quad
&D_t\eta=D_x\te+\Psi+D_x\xi,\\
{\rm(O_\gm)}\quad
&\gm_t=\eta,
\end{aligned}
\right.\\
&\qquad\Psi:=R(\xi,D_x\xi)\xi-R(\xi,D_t\xi)\eta,
\quad\Phi:=(\nr{D_t\xi}^2-\nr{D_x\xi}^2)\xi-R(\xi,\eta)\eta,\\
&\qquad*^\perp:=*-g(*,\xi)\xi.
\end{aligned}
\end{equation}

We can restore the equation of motion \eqref{eq:base-single} from this system.
Namely, the equation of motion \eqref{eq:base-single} and the coupled system \eqref{eq:base-coupled} are equivalent.

\begin{proposition}
\label{proposition:restore-base-single}
Let $\{\gm,\eta,\xi,\te\}$ be a solution to \eqref{eq:base-coupled}.
Assume that functions $\gm$, $\eta$ and $\xi$ are of class $C^1$, and functions $\gm_{tx}=\gm_{xt}$, $\eta_{tx}=\eta_{xt}$, $\xi_{tt}$ and $\xi_{xx}$ are continuous.
If the initial data satisfies $\gm_x=\xi$ and $D_t\xi=D_t\gm_x$ at $t=0$,
then $\gm$ satisfies the equation of motion \eqref{eq:base-single} with a function $\mu$.
Moreover, $\te$, $\xi$ and $\eta$ coincide with those which are obtained by the above procedure starting from $\gm$ and $\mu$.
\end{proposition}

\begin{proof}
We know $D_t=\Dt$ because $\gm_t=\eta$.
For $D_x$, similarly to $\Dx$,
we have $g(D_x\xi,\xi)=0$ and $g(D_x^2\xi,\xi)=-\nr{D_x\xi}^2$,
and
\begin{equation}
\label{eq:(Dt2xi-Dx2xi)perp}
(D_t^2\xi-D_x^2\xi)^\perp
=D_t^2\xi-D_x^2\xi+(\nr{D_t\xi}^2-\nr{D_x\xi}^2)\xi.
\end{equation}
We differentiate $\al:=\xi-\gm_x$ for time direction.
Note that $\Dx*=D_x\!*-\Gm(\al,*)$.
\begin{equation}
\begin{aligned}
&\Dt^2\al=\Dt(\Dt\xi-\Dx\gm_t)
=\Dt(\Dt\xi-\Dx\eta)
=\Dt^2\xi-\Dx\Dt\eta-R(\gm_t,\gm_x)\eta\\
&\qquad=D_t^2\xi-D_xD_t\eta+\Gm(\al,D_t\eta)-R(\gm_t,\gm_x)\eta.
\end{aligned}
\end{equation}
Here, we need $\gm_{tx}=\gm_{xt}$ and $\eta_{tx}=\eta_{xt}$ to show $\Dt\Dx\eta-\Dx\Dt\eta=R(\gm_t,\gm_x)\eta$.
In this equation,
\begin{equation}
\begin{aligned}
D_xD_t\eta&=D_x(D_x\te+\Psi+D_x\xi)
=-\Phi+\te^\perp+D_x^2\xi\\
&=-(\nr{D_t\xi}^2-\nr{D_x\xi}^2)\xi+R(\xi,\eta)\eta
+(D_t^2\xi-D_x^2\xi)^\perp+D_x^2\xi\\
&=-R(\eta,\xi)\eta+D_t^2\xi.
\end{aligned}
\end{equation}
Hence,
\begin{equation}
\Dt^2\al=R(\eta,\al)\eta+\Gm(\al,D_t\eta).
\end{equation}

Therefore, equality $\al=\Dt\al=0$ at $t=0$ implies $\xi=\gm_x$ for all $t$ by the uniqueness of solution to the ODE.
In particular, we have $D_x=\Dx$.
Next, letting $\mu=\nr{\Dx\xi}^2-\nr{\Dt\xi}^2-g(\te,\xi)-1$,
and using $\rm(W_\xi)$ and \eqref{eq:(Dt2xi-Dx2xi)perp}, we get
\begin{equation}
\Dt^2\xi-\Dx^2\xi-\mu\xi=\te^\perp+(g(\te,\xi)+1)\xi
=\te+\xi.
\end{equation}
Then, the original equation is restored as follows.
\begin{equation}
\begin{aligned}
\Dt\gm_t
&=\Dt\eta
=\Dx\te+\Psi+\Dx\xi
=\Dx(\Dt^2\xi-\Dx^2\xi-\mu\xi)+\Psi\\
&=\Dx(\Dt^2\gm_x-\Dx^2\gm_x-\mu\gm_x)+\Psi.
\end{aligned}
\end{equation}

Moreover, functions $\xi$, $\eta$ and $\te$ coincide with those which we defined by the previous procedure.
\end{proof}

\section
{Short time existence}
\label{section:short-time-exisistence}

To solve the coupled system \eqref{eq:base-coupled},
we need to specify clearly meaning of solving each equation.
Firstly, each equation other than $\rm(O_\gm)$ is an equation for a vector field along $\gm(x,t)$,
and each vector field is expressed as an ${\bf R}^n$-valued function via an orthonormal frame field $\{e_i\}$.
Any motion of $\gm$ is expressed by $\Gm$ and $R$, and, for example, continuousness of $\gm(x,t)$ is equivalent to continuousness of $\Gm(x,t)$ and $R(x,t)$.

We use usual notation $C^k$ only for functions $u(x,t)$.
When a function $u$ is of class $C^k$ only in $x$ (resp. $t$),
we denote it explicitly by $C^k(x)$ (resp. $C^k(t)$).
Also, we denote by $\Cls ij$ the class of functions $u(x,t)$ such that $\pd_x^k\pd_t^\ell u\in C^0(x,t)$ for $k\le i$ and $\ell\le j$.

We solve each equation in framework
$\gm\in\Cls01$,
$\xi\in C^1$,
$\eta\in\Cls01$,
$\te\in\Cls10$.
In this framework, we see $\Psi$, $\Phi\in C^0$.
Initial values are assumed $\gm(x,0)\in C^0(x)$, $\xi(x,0)\in C^1(x)$, $\xi_t(x,0)\in C^0(x)$ and $\eta(x,0)\in C^0(x)$.
In equations $\rm(O_\te)$, $\rm(W_\xi)$ and $\rm(O_\eta)$, the function $\gm$ is a known function and $\te$, $\xi$ and $\eta$ are vector fields along $\gm$.
The vector field $\xi$ is assumed to be a unit vector field.

We use following norms of functions.
\begin{equation}
\begin{aligned}
&m_0(u):=\sup_x\nr{u},
\quad m_1(u):=m_0(u)+m_0(u_x)+m_0(u_t),\\
&M_0(u,T):=\sup_{0\le t\le T}m_0(u)(t),\\
&M_{1,0}(u,T):=M_0(u,T)+M_0(u_x,T),
\quad M_{0,1}(u,T):=M_0(u,T)+M_0(u_t,T),\\
&M_1(u,T):=M_0(u,T)+M_0(u_x,T)+M_0(u_t,T).
\end{aligned}
\end{equation}

\subsection
{Estimation for $\rm(O_\eta)$ and $\rm(O_\gm)$}
\label{subsection:O-eta+O-gm}

We consider simple ordinary differential equations $\rm(O_\eta)$ and $\rm(O_\gm)$.
Propositions in this subsection are basic, but their proof give a plan of estimation of solutions to the partial differential equation $\rm(W_\xi)$.
So we include proofs of these propositions.

We give estimation of the solution $u$ to the ordinary differential equation
\begin{equation}
u'=F(t,u),\quad u(0)=a.
\end{equation}
The following lemma is well-known.

\begin{lemma}
\label{lemma:ODE-t-existence}
We consider the ordinary differential equation $u'(t)=F(t,u(t))$ with $u(0)=a$.
If the function $F(t,u)$ is continuous and satisfies $\nr{F(t,u)}\le K$ on the domain $\{(t,u)\mid0\le t\le T_0,\nr{u-a}\le r\}$,
then there exists a unique solution $u(t)\in C^1(t)$ on the time interval $[0,\min\{T_0,r/K\}]$, and satisfies
\begin{equation}
\nr{u-a}\le Kt,
\quad\nr{u'}\le K.
\end{equation}

Moreover, if the equation depends continuously on $x$, namely the equation has form
$u_t(x,t)=F(x,t,u(x,t))$ with $u(x,0)=a(x)$,
and function $F$ and initial value $a$ are continuous also with respect to $x$,
then the solution $u(x,t)$ belongs to $\Cls01$.
\end{lemma}

We study influence of the function $F(t,u)$ to solutions.

\begin{lemma}
\label{lemma:ODE-t-difference}
Suppose that $F_1(t,u)$ and $F_2(t,u)$ are continuous functions,
and each ODE $u_i'=F_i(t,u_i)$ with $u_i(0)=a$ has a $C^1(t)$-solution $u_i$ on the time interval $[0,T]$,
and $\nr{u_i-a}\le r$ holds on the time interval $[0,T]$.
Let $S=\{(t,u)\mid0\le t\le T,\nr{u-a}\le r\}$,
$\dl F=F_2-F_1$ and $\dl u=u_2-u_1$.

If $\nr{(F_1)_u}\le K$ and $\nr{\dl F}\le\ep$ on $S$, then it holds that
\begin{equation}
\nr{\dl u}\le\ep K^{-1}(\exp(Kt)-1),\quad
\nr{\dl u'}\le\ep\exp(Kt).
\end{equation}
\end{lemma}

\begin{proof}
Since $\dl u'=F_2(t,u_2)-F_1(t,u_1)=\dl F(t,u_2)+\{F_1(t,u_2)-F_1(t,u_1)\}$,
\begin{equation}
\begin{aligned}
&\nr{\dl u'}\le\ep+\max\{\nr{(F_1)_u}\}\nr{\dl u}
\le\ep+K\nr{\dl u},\\
&\nr{\dl u}\le\int_0^t\nr{\dl u'(\tau)}\,d\tau
\le\int_0^t\ep+K\nr{\dl u(\tau)}\,d\tau.
\end{aligned}
\end{equation}
By solving this integral inequality, we get
\begin{equation}
\begin{aligned}
\nr{\dl u}&\le\ep K^{-1}(\exp(Kt)-1),\\
\nr{\dl u'}&\le\ep+K\nr{\dl u}\le\ep\exp(Kt).
\end{aligned}
\end{equation}
\end{proof}

To use the contraction mapping theorem, we prepare the following lemma.

\begin{lemma}
\label{lemma:ODE-t-contraction}
Suppose that functions $F_1(t,u)$ and $F_2(t,u)$ are continuous,
and each ODE $u_i'=F_i(t,u_i)$ with $u_i(0)=a$ has a $C^1(t)$-solution $u_i$ on the time interval $[0,T]$,
and $\nr{u_i-a}\le r$ on $[0,T]$.
Let $S=\{(t,u)\mid0\le t\le T,\nr{u-a}\le r\}$,
$\dl F:=F_2-F_1$ and $\dl u:=u_2-u_1$.

If $\nr{(F_1)_u}\le K$ and $\nr{\dl F_t}\le\ep$ on $S$,
and if $\dl F(0,u)=0$, then it holds that
\begin{equation}
\begin{aligned}
\nr{\dl u}&\le\ep K^{-2}(\exp(Kt)-Kt-1),\\
\nr{\dl u'}&\le\ep K^{-1}(\exp(Kt)-1).
\end{aligned}
\end{equation}
\end{lemma}

\begin{proof}
For $\dl u'=F_2(t,u_2)-F_1(t,u_1)=\dl F(t,u_2)+\{F_1(t,u_2)-F_1(t,u_1)\}$,
we have
\begin{equation}
\begin{aligned}
&\nr{\dl F(t,u_2)}
\le\int_0^t\nr{\dl F_t(\tau,u_2)}\,d\tau
\le\ep t,\\
&\nr{F_1(t,u_2)-F_1(t,u_1)}
\le K\nr{\dl u}.
\end{aligned}
\end{equation}
Hence,
\begin{equation}
\nr{\dl u(t)}\le\int_0^t\nr{\dl u'(\tau)}\,d\tau
\le\int_0^t\ep t+K\nr{\dl u(\tau)}\,d\tau
=\frac12\ep t^2+K\int_0^t\nr{\dl u(\tau)}\,d\tau.
\end{equation}
By solving this integral inequality, we get
\begin{equation}
\begin{aligned}
\nr{\dl u}&\le\ep K^{-2}(\exp(Kt)-Kt-1),\\
\nr{\dl u'}&\le\ep t+K\nr{\dl u(t)}
\le\ep K^{-1}(\exp(Kt)-1).
\end{aligned}
\end{equation}
\end{proof}

We treat equation $\rm(O_\eta)$
$D_t\eta=D_x\te+\Psi+D_x\xi$,
$\Psi=R(\xi,D_x\xi)\xi-R(\xi,D_t\xi)\eta$,
$\eta(x,0)=a(x)$
in the following form.
\begin{equation}
\label{eq:O-eta-local-expression}
\eta_t+\Gm(\eta,\eta)+R(\xi,\xi_t+\Gm(\eta,\xi))\eta=D_x\te+D_x\xi.
\end{equation}
I.e., we define the function $F(x,t,u)$ in Lemma \ref{lemma:ODE-t-existence} by
\begin{equation}
F(x,t,u)=-\Gm(u,u)-R(\xi,\xi_t+\Gm(u,\xi))u+D_x\te+D_x\xi.
\end{equation}

By Lemma \ref{lemma:ODE-t-existence}, this ODE has a solution $\eta=u\in\Cls01$ for known functions $\gm\in C^0$, $\xi\in C^1$ and $\te\in\Cls10$, and a initial value $u(x,0)=a(x)\in C^0(x)$, as follows.

\begin{proposition}
\label{proposition:O-eta-existence}
We consider the ODE $\rm(O_\eta)$ on the time interval $[0,T_0]$.
We assume that known functions
$\gm\in C^0$, $\xi\in C^1$ and $\te\in\Cls10$ satisfies
$M_0(\Gm,T_0)$, $M_0(R,T_0)$, $M_1(\xi,T_0)$, $M_{1,0}(\te,T_0)\le K_0$,
and that the initial value satisfies $m_0(a)\le r_0$.
Then, for any $r>r_0$, there exists a positive constant $K_1$ depending only on $K_0$ and $r$,
such that a unique solution $\eta\in\Cls01$ exists on the time interval $[0,T=\min\{T_0,(r-r_0)/K_1\}]$,
and the solution satisfies $M_0(\eta,T)\le r$ and $M_0(\eta_t,T)\le K_1$.
\end{proposition}

\begin{proof}
The function $F$ is a second order polynomial in $u=\eta$,
and its coefficients are all continuous and bounded from above using $K_0$.
Therefore, there exists a positive constant $K_1$ depending only on $r$ and $K_0$
such that $\max_{S(T_0,r)}\{\nr F\}\le K_1$.
Hence, by Lemma \ref{lemma:ODE-t-existence},
there exists a solution $\eta\in\Cls01$ on the time interval $[0,T]$,
and it satisfies
$m_0(\eta)\le r$ and $m_0(\eta_t)\le K_1$.
\end{proof}

Moreover, the difference of solutions corresponding to 2 sets of known functions is estimated as follows.

\begin{proposition}
\label{proposition:O-eta-difference}
We consider on the time interval $[0,T]$.
Let $i$ be an index $1$ or $2$.
We assume that, for each $i$,
functions $\gm_i\in C^0$, $\xi_i\in C^1$ and $\te_i\in\Cls10$ satisfy
$M_0(\Gm_i,T)$, $M_0(R_i,T)$, $M_1(\xi_i,T)$, $M_{1,0}(\te_i,T)\le K_0$,
and that there exists a solution $\eta_i$ to the ODE $\rm(O_\eta)$ with initial value $a$,
and it satisfies $M_0(\eta_i)\le r$.
We denote by $\dl p$ the difference $p_2-p_1$ of functions,
and assume that
$M_0(\dl\Gm,T)$, $M_0(\dl R,T)$, $M_1(\dl\xi,T)$, $M_{1,0}(\dl\te,T)\le\ep$.
Then, there exists a positive constant $K$ depending only on $K_0$ and $r$ such that $\dl\eta$ satisfies
$M_0(\dl\eta,T)\le KT\ep$ and $M_0(\dl\eta_t,T)\le K\ep$.
\end{proposition}

\begin{proof}
Since $F$ is a second order polynomial in $u=\eta$, $\nr{(F_1)_u}$ is bounded from above by a positive constant $K_1$ depending only on $K_0$ and $r$.
The difference $\dl F$ is a second order polynomial in $\dl u$.
When we replace all $p_2$ in its coefficients by $p_1+\dl p$ and expand them, each term contains at least one $\dl p$.
Therefore, there exists a positive constant $K_2$ depending only on $K_0$ and $r$
such that $\dl F$ satisfies $M_0(\dl F,T)\le K_2\ep$.
Applying Lemma \ref{lemma:ODE-t-difference} with $K:=K_1$ and $\ep:=K_2\ep$,
we have $M_{0,1}(\dl\eta,T)\le K_3\ep$ using $K_3=K_2\max\{K_0^{-1}(\exp K_0-1),\exp K_0\}$.
\end{proof}

We express the ODE $\rm(O_\gm)$
$\gm_t=\eta$ by a local coordinate system.
At $\gm(x,t)\in M$, we express the conversion matrix of $e_i$ and $\pd/\pd x^i$ by $(h^i{}_j(\gm(x,t)))$.
I.e., $e_j=h^i{}_j\cdot\pd/\pd x^i$.
Then, $\gm^i_t\cdot\pd/\pd x^i=\eta^je_j=\eta^jh^i{}_j\cdot\pd/\pd x^i$,
and the ODE becomes $\gm^i_t(x,t)=h^i{}_j(\gm(x,t))\eta^j(x,t)$.
Since $h^i{}_j$ is smooth, when we express $\gm_t(x,t)=F(x,t,\gm)$,
$F_t$ is continuous if $\eta\in\Cls01$.
From this we have the following propositions about the existence of solutions and the continuousness of solutions with respect to coefficients.

\begin{proposition}
\label{proposition:O-gm-sol}
We consider the ODE $\rm(O_\gm)$: $\gm^i_t=h^i{}_j(\gm)\eta^j$ with $\gm(x,0)=a(x)$ on the time interval $[0,T_0]$.
We assume that the function $h^i{}_j(u)$ on the manifold satisfies
$\nr{h^i{}_j(u)}$, $\nr{(h^i{}_j)_u(u)}\le K_0$
on $\cup_x\{u\mid\nr{u-a(x)}\le r\}$,
and the known functions satisfy
$\eta\in\Cls01$ and $M_{0,1}(\eta,T_0)\le K_0$.
Then, there exists a positive constant $K_1$ depending only on $K_0$
such that a solution $\gm\in\Cls01$ uniquely exists on the time interval $[0,T]$
and satisfies $M_0(\gm-a,T)\le r$ and $M_0(\gm_t)\le K_1$, where $T:=\min\{T_0,r/K_1\}$.
\end{proposition}

\begin{proof}
It is similar to Proof of Proposition \ref{proposition:O-eta-existence}.
The function $F(x,t,u)=h^i{}_j(u)\eta^j(x,t)$ is continuous, and there is a positive constant $K_1$ depending only on $K_0$ such that $\nr{F}\le K_1$ on the time interval $[0,1]$.
Therefore, by Lemma \ref{lemma:ODE-t-existence}, there exists a solution $\gm\in\Cls01$ on the time interval $[0,T]$ such that $m_0(\gm-a)\le r$ and $m_0(\gm_t)\le K_1$.
\end{proof}

\begin{proposition}
\label{proposition:O-gm-difference}
We consider the ODE $\rm(O_\gm)$ in Proposition \ref{proposition:O-gm-sol}
for $2$ known functions $\eta=\eta_k$ $(k=1,2)$,
and assume that each has a solution $\gm_k\in\Cls01$ on the time interval $[0,T]$.
We also assume that $h^i{}_j(u)$ satisfies $\nr{h^i{}_j(u)}$, $\nr{(h^i{}_j)_u(u)}\le K_0$ on $\cup_x\{u\mid\nr{u-a(x)}\le r\}$,
and that $\eta_1(x,0)=\eta_2(x,0)$, $\eta_k\in\Cls01$ and $M_{0,1}(\eta_k,T)\le K_0$.
And, we assume that the solution $\eta_k$ satisfies $M_0(\gm_k-a,T)\le r$ and $M_0(\gm_t)\le K_1$.
We denote by $\dl*=*_2-*_1$ the difference of functions,
and assume that $M_{0,1}(\dl\eta,T)\le\ep$.
Then, there exists a positive constant $K_1$ depending only on $K_0$
such that $M_{0,1}(\dl\gm,T)\le K_1T\ep$.
\end{proposition}

\begin{proof}
From the assumption, we have $\dl F(x,0,u)=0$, and there is a positive constant $K_1$ depending only on $K_0$ such that $\nr{F_u}\le K_1$ and $\nr{F_t}\le K_1\ep$ on the time interval $[0,T]$.
Therefore, from Lemma \ref{lemma:ODE-t-contraction},
we have $\nr{\dl\gm(x,t)}\le\ep K_2t$ and $\nr{\dl\gm_t(x,t)}\le\ep K_2t$
where $K_2=\max\{K_1^{-1}(\exp(K_1)-K_1-1),\exp(K_1)-1\}$.
\end{proof}

\subsection
{Estimation for $\rm(W_\xi)$}
\label{subsection:W-xi}

We transform equation $\rm(W_\xi)$ $(D_t^2\xi-D_x^2\xi)^\perp=\te^\perp$ as follows.
This transformation is justified under the assumption $\xi\in C^2$.
Because $\nr\xi^2=1$,
we have
$g(D_x^2\xi,\xi)=-\nr{D_x\xi}^2$,
$(D_x^2\xi)^\perp=D_x^2\xi+\nr{D_x\xi}^2\xi$.
Similarly, we have
$(D_t^2\xi)^\perp=D_t^2\xi+\nr{D_t\xi}^2\xi$.
Therefor equation $\rm(W_\xi)$ is transformed to
\begin{equation}
\label{eq:xi-wave-rep}
D_t^2\xi-D_x^2\xi=(\nr{D_x\xi}^2-\nr{D_t\xi}^2)\xi+\te-g(\te,\xi)\xi.
\end{equation}

Equation \eqref{eq:xi-wave-rep} is an equation for $\xi(x,t)$, for given $\gm(x,t)$ and $\te(x,t)$.
It can be written as a semilinear hyperbolic equation
\begin{equation}
\xi_{tt}-\xi_{xx}=\text{(lower order non-linear terms)}f+h_x
\end{equation}
using a local coordinate system.
We use the term $h_x$ when we cannot estimate $h_x$ but have bounds of $h_t$.
And, the semilinear hyperbolic equation is transformed to an integral equation
\begin{equation}
\label{eq:wave-integral-equation}
\begin{aligned}
\xi(x,t)
&=\frac12\{a(x+t)+a(x-t)\}+\frac12\int_{x-t}^{x+t}b(y)\,dy
+\frac12\int_0^t\int_{x-(t-\tau)}^{x+(t-\tau)}f(y,\tau)\,dyd\tau\\
&\qquad+\frac12\int_0^th(x+(t-\tau),\tau)-h(x-(t-\tau),\tau)\,d\tau.
\end{aligned}
\end{equation}
The functions $a$ and $b$ are initial values: $a(x)=\xi(x,0)$ and $b(x)=\xi_t(x,0)$.

We write $I(a,b,f,h)$ the right hand side of \eqref{eq:wave-integral-equation}.
We need some caution to show regularity of solutions to the equation $u=I(a,b,f,h)$,
because we assume only $\gm\in\Cls01$.
In addition, the variable $x$ runs in ${\bf R}$ and functions are periodic with period $1$ with respect to $x$ in this expression.

The term of integration of $h$ becomes, by integration by substitution: $x\pm(t-\tau)=y$,
\begin{equation}
\int_0^th(x\pm(t-\tau),\tau)\,d\tau
=\mp\int_{x\pm t}^xh(y,t\mp(y-x))\,dy,
\end{equation}
and,
\begin{align}
&\begin{aligned}
&\pd_x\int_0^th(x\pm(t-\tau),\tau)\,d\tau
=\mp\{h(x,t)-h(x\pm t,0)\pm\int_{x\pm t}^xh_t(y,t\mp(y-x))\,dy\}\\
&\qquad=\pm\{h(x\pm t,0)-h(x,t)\}\pm\int_0^th_t(x\pm(t-\tau),\tau)\,d\tau,
\end{aligned}\\
&\begin{aligned}
&\pd_t\int_0^th(x\pm(t-\tau),\tau)\,d\tau
=\mp\{\mp h(x\pm t,0)+\int_{x\pm t}^xh_t(y,t\mp(y-x))\,dy\}\\
&\qquad=h(x\pm t,0)+\int_0^th_t(x\pm(t-\tau),\tau)\,d\tau.
\end{aligned}
\end{align}
Therefore, by defining
\begin{equation}
2\wh h(x,t):=h(x+t,0)+h(x-t,0)-2h(x,t),
\end{equation}
we have
\begin{align}
&\begin{aligned}
2u_x(x,t)&=\{a'(x+t)+a'(x-t)\}+\{b(x+t)-b(x-t)\}\\
&\qquad+\int_0^tf(x+(t-\tau),\tau)-f(x-(t-\tau),\tau)\,d\tau\\
&\qquad+2\wh h(x,t)
+\int_0^th_t(x+(t-\tau),\tau)+h_t(x-(t-\tau),\tau)\,d\tau,\\
\end{aligned}\\
&\begin{aligned}
2u_t(x,t)&=\{a'(x+t)-a'(x-t)\}+\{b(x+t)+b(x-t)\}\\
&\qquad+\int_0^tf(x+(t-\tau),\tau)+f(x-(t-\tau),\tau)\,d\tau\\
&\qquad+\{h(x+t,0)-h(x-t,0)\}\\
&\qquad+\int_0^th_t(x+(t-\tau),\tau)-h_t(x-(t-\tau),\tau)\,d\tau.
\end{aligned}
\end{align}

Thus we get the following formula.
\begin{lemma}
\label{lemma:I(a,b,f,h)-C1-formula}
Let $u^\pm(x,t):=u_x(x,t)\pm u_t(x,t)$.
Then we have
\begin{equation}
\begin{aligned}
u^\pm(x,t)&=a'(x\pm t)\pm b(x\pm t)
\pm\int_0^tf(x\pm(t-\tau),\tau)\,d\tau\\
&\quad+\{h(x\pm t,0)-h(x,t)\}
+\int_0^th_t(x\pm(t-\tau),\tau)\,d\tau.
\end{aligned}
\end{equation}
\end{lemma}

We can solve the semilinear integral equation \eqref{eq:wave-integral-equation} by a standard contraction method.
However, we explain details of it for completeness and preparation of estimation of $C^2$-norm.
Also, we use the following norm for initial data.
\begin{equation}
m_0(a,a',b):=m_0(a)+m_0(a')+m_0(b).
\end{equation}

By Lemma \ref{lemma:I(a,b,f,h)-C1-formula}, we have
\begin{equation}
\begin{aligned}
m_0(u)(t)&\le m_0(a)+tm_0(b)
	+\int_0^t(t-\tau)m_0(f)(\tau)\,d\tau
	+\int_0^tm_0(h)(\tau)\,d\tau,\\
m_0(u_x)(t)&\le m_0(a')+m_0(b)
	+\int_0^tm_0(f)(\tau)\,d\tau
	+m_0(\wh h)(t)
	+\int_0^tm_0(h_t)(\tau)\,d\tau,\\
m_0(u_t)(t)&\le m_0(a')+m_0(b)
	+\int_0^tm_0(f)(\tau)\,d\tau
	+\int_0^tm_0(h_t)(\tau)\,d\tau.
\end{aligned}
\end{equation}
Summing up them, we get the following estimation.

\begin{lemma}
\label{lemma:I(a,b,f,h)-C1-estimation}
The function $u=I(a,b,f,h)$ is estimated as follows.
\begin{equation}
\begin{aligned}
m_1(u)(t)&\le(2+t)m_0(a,a',b)+(2+t)\int_0^tm_0(f)(\tau)\,d\tau\\
&\qquad+\int_0^tm_0(h)(\tau)\,d\tau
+m_0(\wh h)(t)+2\int_0^tm_0(h_t)(\tau)\,d\tau,\\
M_1(u,T)&\le(2+T)m_0(a,a',b)+T(2+T)M_0(f,T)\\
&\qquad+TM_0(h,T)+M_0(\wh h,T)+2TM_0(h_t,T).
\end{aligned}
\end{equation}
\end{lemma}

We gather estimations of $M_0(\wh h,T)$ which we need later.
We define domain $S=S(T,r):=\{(x,t,u)\mid0\le t\le T,\nr u\le r\}$
and $S=S(T,r):=\{(x,t,u,p,q)\mid0\le t\le T,\nr u+\nr p+\nr q\le r\}$.
For continuous functions $H(x,t,u)$ and $F(x,t,u,p,q)$, we define
\begin{equation}
\begin{aligned}
\max_S\{\nr H\}&:=\max\{\nr{H(x,t,u)}\mid(x,t,u)\in S\},\\
\max_S\{\nr F\}&:=\max\{\nr{H(x,t,u,p,q)}\mid(x,t,u,p,q)\in S\}.
\end{aligned}
\end{equation}
And, for plural functions $F_1$, $F_2$, etc.,
we write shortly
\begin{equation}
\max_S\{\nr{F_1},\nr{F_2}\}:=\max\{\max_S\{\nr{F_1}\},\max_S\{\nr{F_2}\}\},
\end{equation}
etc.
For a continuous function $H(x,t,u)$, we define
\begin{equation}
Z(t,r):=\max\{\nr{H(x_1,0,u)-H(x_2,0,u)}\mid\nr{x_1-x_2}\le t,\nr{u}\le r\}
\end{equation}
Note that $Z(0,r)=0$ and that $Z$ depends only on $H(x,0,u)$.
If $H(x,0,u)$ is differentiable with respect to $x$, then $Z(t,r)\le\max\{\nr{H_x(x,0,u)}\mid\nr u\le r\}\cdot t$.
And, we define
\begin{equation}
2\wh H(x,t,u):=H(x+t,0,u)+H(x-t,0,u)-2H(x,t,u).
\end{equation}
Moreover, for $u=u(x,t)$, we define
\begin{equation}
\begin{aligned}
H(u)(x,t)&:=H(x,t,u(x,t)),\\
\wh H(u)(x,t)&:=H(u)(x+t,0)+H(u)(x-t,0)-2H(u)(x,t)
\end{aligned}
\end{equation}
Then,
\begin{equation}
\begin{aligned}
2\wh H(x,t,u)&=\{H(x+t,0,u)-H(x,0,u)\}+\{H(x-t,0,u)-H(x,0,u)\}\\
&\qquad-2\{H(x,t,u)-H(x,0,u)\},
\end{aligned}
\end{equation}
and we have the following estimation.

\begin{lemma}
\label{lemma:htH-Z-estimation}
If the function $H(x,t,u)$ is continuous and $H_t$ is continuous,
then
\begin{equation}
\nr{\wh H(x,t,u)}\le Z(t,\nr u)+t\max_{S(t,\nr u)}\{\nr{H_t}\}.
\end{equation}
\end{lemma}

Now preparations to prove existence of solutions to the semilinear integral wave equation $u=I(a,b,F(u),H(u))$ are complete.

\begin{lemma}
\label{lemma:wave-eq-C1-existence}
Assume that the functions $f$ and $h$ are expressed as
$f(x,t)=F(u):=F(x,t,u,u_x,u_t)$ and $h(x,t)=H(u):=H(x,t,u)$ respectively,
where functions $F$ and $H$ are continuous and $F_u$, $F_p$, $F_q$, $H_t$, $H_u$, $H_{tu}$ and $H_{uu}$ are continuous.
If initial value satisfies $a\in C^1(x)$ and $b\in C^0(x)$,
then the semilinear integral equation $u=I(a,b,F(u),H(u))$ has a unique short-time solution $u(x,t)\in C^1$.
\end{lemma}

\begin{proof}
We choose a real number $r_0>2m_0(a,a',b)$, and consider on the domain $S(T_0,r_0)$ hereinafter.
We assume that
\begin{equation}
\max_{S(T_0,r_0)}\{\nr{F},\nr{F_u},\nr{F_p},\nr{F_q},\nr{H},\nr{H_t},\nr{H_u},\nr{H_{tu}},\nr{H_{uu}}\}
\le K_0.
\end{equation}
And, using $Z(t,r)$ in Lemma\ref{lemma:htH-Z-estimation}, we estimate $\wh H$ as
\begin{equation}
\nr{\wh H(x,t,u)}
\le Z(t,r_0)+t\max_{S(t,r_0)}\{\nr{H_t}\}
\le Z(t,r_0)+tK_0.
\end{equation}

For $u(x,t)$ satisfying $M_1(u,T_0)\le r_0$, let $f(x,t):=F(u)$ and $h(x,t):=H(u)$.
Then we have
\begin{equation}
\begin{aligned}
m_0(f)(t)&\le\max_{S(t,m_1(u)(t))}\{\nr F\}
\le K_0,\\
m_0(h)(t)&\le\max_{S(t,m_0(u)(t))}\{\nr H\}
\le K_0,\\
m_0(h_t)(t)&=m_0(H_t+H_uu_t)(t)
\le\max_{S(t,m_0(u)(t))}\{\nr{H_t},\nr{H_u}\}(1+m_1(u)(t))\\
&\le K_0(1+r_0).
\end{aligned}
\end{equation}
And,
\begin{equation}
\begin{aligned}
2\wh h(x,t)&=H(x+t,0,u(x+t,0))+H(x-t,0,u(x-t,0))
-2H(x,t,u(x,t))\\
&=2\wh H(x,t,u(x,t))+(H(x+t,0,u(x+t,0))-H(x+t,0,u(x,t)))\\
&\qquad\qquad+(H(x-t,0,u(x-t,0))-H(x-t,0,u(x,t))).
\end{aligned}
\end{equation}

Here, because
\begin{equation}
\begin{aligned}
&m_0(\wh H(x,t,u(x,t)))
\le Z(t,r_0)+tK_0,\\
&\nr{H(x\pm t,0,u(x\pm t,0))-H(x\pm t,0,u(x,t))}\\
&\qquad\le\max_{S(0,M_0(u,t))}\{\nr{H_u}\}\cdot\nr{u(x\pm t,0)-u(x,t)}\\
&\qquad\le K_0\cdot(M_0(u_t,t)+M_0(u_x,t))t
\le tK_0r_0,
\end{aligned}
\end{equation}
we get
\begin{equation}
m_0(\wh h)
\le Z(t,r_0)+tK_0+tK_0r_0
=Z(t,r_0)+tK_0(1+r_0).
\end{equation}

Therefore,
\begin{equation}
\begin{aligned}
&M_0(f,T),M_0(h,T)\le K_0,\quad M_0(h_t,T)\le K_0(1+r_0),\\
&M_0(\wh h,T)\le Z(T,r_0)+TK_0(1+r_0).
\end{aligned}
\end{equation}
And, by Lemma \ref{lemma:I(a,b,f,h)-C1-estimation},
$w(x,t):=I(a,b,F(u),H(u))$ satisfies
\begin{equation}
\label{eq:M1(w,T)-estimate}
\begin{aligned}
M_1(w,T)&\le(2+T)m_0(a,a',b)+T(2+T)K_0+TK_0+2TK_0(1+r_0)\\
&\qquad+Z(T,r_0)+TK_0(1+r_0)\\
&\le(2+T)m_0(a,a',b)+T(6+T)K_0(1+r_0)+Z(T,r_0).
\end{aligned}
\end{equation}

Let $\al(t)$ be a monotone increasing function with respect to $t$
defined by
\begin{equation}
\al(t):=(2+t)m_0(a,a',b)+t(6+t)K_0(1+r_0)+Z(t,r_0),
\end{equation}
and we express the last side of \eqref{eq:M1(w,T)-estimate} by $\al(T)$.
We choose $u_0(x,t)$ as
\begin{equation}
u_0(x,t):=a(x)+\frac12\int_{x-t}^{x+t}b(y)\,dy,
\end{equation}
which satisfies the initial condition.
Since
$\al(0)=2m_0(a,a',b)<r_0$ and $M_1(u_0,0)=m_0(a,a',b)<r_0$,
we can choose $0<T_1\le T_0$ such that
$\al(T_1)$, $M_1(u_0,T_1)\le r_0$.
We define a set of functions $\mathcal S$ on $S^1\times[0,T_1]$ by
$\mathcal S:=\{u\in C^1\mid\text{$u$ satisfies the initial condition, } M_1(u,T_1)\le r_0\}$.
Since $M_1(u_0,T_1)\le r_0$, the set $\mathcal S$ is not empty.
From the above estimation, $w=\Lm(u):=I(a,b,F(u),H(u))$ satisfies $M_1(w,T_1)\le\al(T_1)\le r_0$ for $u\in\mathcal S$,
and $\Lm$ is a mapping from $\mathcal S$ to $\mathcal S$.

For $u_1$, $u_2\in\mathcal S$, let
$w_i=\Lm(u_i)$, $\dl u=u_2-u_1$, $\dl w=w_2-w_1$,
$\dl f=F(u_2)-F(u_1)$ and $\dl h=H(u_2)-H(u_1)$.
Then we have $\dl w=I(0,0,\dl f,\dl h)$.
And,
\begin{equation}
\begin{aligned}
&\nr{\dl f}\le\max_{S(t,r_0)}\{\nr{F_u},\nr{F_p},\nr{F_q}\}
	(\nr{\dl u}+\nr{\dl u_x}+\nr{\dl u_t})
\le K_0m_1(\dl u),\\
&\nr{\dl h}\le\max_{S(t,r_0)}\{\nr{H_u}\}\nr{\dl u}
\le K_0m_1(\dl u).
\end{aligned}
\end{equation}

Moreover, because
\begin{equation}
\begin{aligned}
\dl h_t
&=H_t(u_2)+H_u(u_2)(u_2)_t-H_t(u_1)-H_u(u_1)(u_1)_t\\
&=\{H_t(u_2)-H_t(u_1)\}+\{H_u(u_2)(u_2)_t-H_u(u_1)(u_2)_t\}+H_u(u_1)(\dl u_t),
\end{aligned}
\end{equation}
we get
\begin{equation}
\begin{aligned}
\nr{\dl h_t}
&\le\max_{S(t,r_0)}\{\nr{H_{tu}},\nr{H_{uu}},\nr{H_u}\}(\nr{\dl u}+\nr{\dl u}\nr{(u_2)_t}+\nr{\dl u_t})
\le K_0(1+r_0)m_1(\dl u).
\end{aligned}
\end{equation}

We summarize the above into estimations
$\nr{\dl f}$, $\nr{\dl h}$, $\nr{\dl h_t}\le K_0(1+r_0)m_1(\dl u)$.
For $\dl\wh h$, since $u_1$ and $u_2$ have same initial value, we get
\begin{equation}
\begin{aligned}
&\dl\wh h(x,t)=-H(x,t,u_2(x,t))+H(x,t,u_1(x,t)),\\
&\nr{\dl\wh h}\le\max_{S(t,r_0)}\{\nr{H_u}\}\cdot\nr{\dl u(x,t)}
\le K_0\cdot M_0(\dl u_t,t)t
\le K_0M_1(\dl u,t)t.
\end{aligned}
\end{equation}

Hence, by Lemma \ref{lemma:I(a,b,f,h)-C1-estimation}, we have
\begin{equation}
\begin{aligned}
M_1(\dl w,T)&\le T(2+T)M_0(\dl f,T)
+TM_0(\dl h,T)+2TM_0(\dl h_t,T)+M_0(\dl\wh h,T)\\
&\le K_0(1+r_0)(T(2+T)+T+2T+T)M_1(\dl u,T)\\
&=K_0(1+r_0)T(6+T)M_1(\dl u,T).
\end{aligned}
\end{equation}
Now, we replace $T_1$ so that $K_0(1+r_0)T_1(6+T_1)<1$ holds.
Then, $\Lm:\mathcal S\to\mathcal S$ is a contraction map with respect to the norm $M_1$, and has a unique fixed point, i.e., a solution $u$.
\end{proof}

For later use, we give a refinement of this lemma.

\begin{lemma}
\label{lemma:wave-eq-M1-existence-uHC1}
Under assumption of Lemma \ref{lemma:wave-eq-C1-existence}, we also assume that $H(x,0,u)$ is differentiable with respect to $x$ and that the following conditions are satisfied.
The initial value satisfies $2m_0(a,a',b)<r_0$, and it holds that
\begin{equation}
\max_{S(T_0,r_0)}\{\nr{F},\nr{F_u},\nr{F_p},\nr{F_q},\nr{H},\nr{H_t},\nr{H_u},\nr{H_{tu}},\nr{H_{uu}}\},
\max_{S(0,r_0)}\{\nr{H_x}\}
\le K_0.
\end{equation}

Then, there exists a positive constant $T_1\le T_0$ depending only on
$m_0(a,a',b)$, $T_0$, $r_0$ and $K_0$,
such that the semilinear integral equation $u=I(a,b,F(u),H(u))$ has a unique solution $u(x,t)\in C^1$ on the time interval $[0,T_1]$ and it satisfies $M_1(u,T_1)\le r_0$.
\end{lemma}

\begin{proof}
We can replace $Z(t,r_0)$ in Proof of Lemma \ref{lemma:wave-eq-C1-existence} by $K_0t$.
\end{proof}

Next, we study dependency of the solution $u$ on the functions $F$ and $H$.

\begin{lemma}
\label{lemma:wave-eq-C1-difference}
For indices $i=1,2$, let $F_i$ and $H_i$ be functions satisfying the condition of Lemma \ref{lemma:wave-eq-C1-existence}.
Assume that there exists a solution $u_i\in C^1$ on the time interval $[0,T]$,
and it satisfies $M_1(u_i,T)\le K$.
Assume also that their initial values are same and $H_1(x,0,u)=H_2(x,0,u)$.
We use notation: $\dl*=*_2-*_1$.
If, on $S:=S(T,K)$,
\begin{equation}
\begin{aligned}
&\max_S\{\nr{(F_1)_u},\nr{(F_1)_p},\nr{(F_1)_q},\nr{(H_1)_u},\nr{(H_1)_{tu}},\nr{(H_1)_{uu}}\}\le K,\\
&\max_S\{\nr{\dl F},\nr{\dl H},\nr{\dl H_t},\nr{\dl H_u}\}\le\ep
\end{aligned}
\end{equation}
then there exists a positive constant $C$ depending monotone increasingly only on $T$ and $K$ such that
\begin{equation}
M_1(\dl u,T)\le CT\ep.
\end{equation}
\end{lemma}

\begin{proof}
Let $f_k=F_k(u_k)$ and $h_k=H_k(u_k)$.
Then, $\dl u=I(\dl a,\dl b,\dl f,\dl h)$, and
\begin{equation}
\begin{aligned}
&\dl f=F_2(u_2)-F_1(u_1)
=(\dl F)(u_2)+\{F_1(u_2)-F_1(u_1)\},\\
&\begin{aligned}
m_0(\dl f)
&\le\max_S\{\nr{\dl F}\}
+\max_S\{\nr{(F_1)_u},\nr{(F_1)_p},\nr{(F_1)_q}\}\cdot m_1(\dl u)
\le\ep+Km_1(\dl u).
\end{aligned}
\end{aligned}
\end{equation}
Similarly, from
\begin{equation}
\dl h=H_2(u_2)-H_1(u_1)
=(\dl H)(u_2)+\{H_1(u_2)-H_1(u_1)\},
\end{equation}
we get
\begin{equation}
m_0(\dl h)\le\max_S\{\nr{\dl H}\}+\max_S\{\nr{(H_1)_u}\}\cdot m_0(\dl u)
\le\ep+Km_0(\dl u).
\end{equation}

Since
\begin{equation}
\begin{aligned}
\dl h_t&=\{(H_2)_t(u_2)+(H_2)_u(u_2)(u_2)_t\}
-\{(H_1)_t(u_1)+(H_1)_u(u_1)(u_1)_t\}\\
&=(\dl H_t)(u_2)+(\dl H_u)(u_2)(u_2)_t+\{(H_1)_t(u_2)-(H_1)_t(u_1)\}\\
&\qquad+((H_1)_u(u_2)-(H_1)_u(u_1))(u_2)_t
+(H_1)_u(u_1)\dl u_t,
\end{aligned}
\end{equation}
we have
\begin{equation}
\begin{aligned}
m_0(\dl h_t)
&\le\max_S\{\nr{\dl H_t}\}+\max_S\{\nr{\dl H_u}\}\cdot m_1(u_2)
+\max_S\{\nr{(H_1)_{tu}}\}\cdot m_0(\dl u)\\
&\qquad+\max_S\{\nr{(H_1)_{uu}}\}\cdot m_0(\dl u)m_1(u_2)
+\max_S\{\nr{(H_1)_u}\}\cdot m_1(\dl u)\\
&\le\ep(1+K)+(2K+K^2)\cdot m_1(\dl u).
\end{aligned}
\end{equation}

Gathering the above and using a positive constant $K_1$ depending only on $K$,
we estimate as
$m_0(\dl f)$, $m_0(\dl h)$, $m_0(\dl h_t)\le K_1\ep+K_1m_1(\dl u)$.
For $\dl\wh h$, we have
\begin{equation}
\begin{aligned}
&\dl\wh h=\wh H_2(u_2)-\wh H_1(u_1)
=\dl\wh H(u_2)+\{\wh H_1(u_2)-\wh H_1(u_1)\},\\
&\dl\wh H(u_2)=-\dl H(x,t,u_2(x,t)),\\
&\nr{\dl\wh H(u_2)}\le\max_S\{\nr{\dl H_t}\}\cdot t
\le\ep t.
\end{aligned}
\end{equation}
And, from
\begin{equation}
\begin{aligned}
&2\nr{\wh H_1(u_2)-\wh H_1(u_1)}
\le\nr{H_1(x+t,0,u_2(x+t,0))-H_1(x+t,0,u_1(x+t,0))}\\
&\qquad\qquad+\nr{H_1(x-t,0,u_2(x-t,0))-H_1(x-t,0,u_1(x-t,0))}\\
&\qquad\qquad+2\nr{H_1(x,t,u_2(x,t))-H_1(x,t,u_1(x,t))}\\
&\qquad\le2\max_S\{\nr{(H_1)_u}\}m_0(\dl a)+2\max_S\{\nr{(H_1)_u}\}m_0(\dl u),
\end{aligned}
\end{equation}
we have
\begin{equation}
\begin{aligned}
m_0(\dl\wh h)
&\ep t+Km_0(\dl u)
\le\ep t+Km_0(\int_0^tm_0(\dl u_t)(\tau)\,d\tau)
\le\ep t+K\int_0^tm_1(\dl u)(\tau)\,d\tau.
\end{aligned}
\end{equation}

Then, by Lemma \ref{lemma:I(a,b,f,h)-C1-estimation}, it holds that
\begin{equation}
\begin{aligned}
m_1(\dl u)(t)&\le m_0(\dl\wh h)(t)
+\int_0^t(2+t)m_0(\dl f)(\tau)+m_0(\dl h)(\tau)+2m_0(\dl h_t)(\tau)\,d\tau\\
&\le T(6+T)K_1\ep
+(K+(5+T)K_1)\int_0^tm_1(\dl u)(\tau)\,d\tau.
\end{aligned}
\end{equation}

We express this inequality as
\begin{equation}
m_1(\dl u)(t)\le\ep_2+K_2\int_0^tm_1(\dl u)(\tau)\,d\tau
\end{equation}
and solve the integral inequality, and get
\begin{equation}
m_1(\dl u)(t)\le\ep_2\exp(K_2t),
\end{equation}
which is the conclusion.
\end{proof}

To apply above lemmas, we write explicitly the functions $f=F(\xi)$ and $h=H(\xi)$ in the semilinear integral equation transformed from Equation $\rm(W_\xi)$.
Since
\begin{equation}
\begin{aligned}
D_x^2\xi&=\pd_x(D_x\xi)+\Gm(\xi,D_x\xi)
=\xi_{xx}+\pd_x(\Gm(\xi,\xi))+\Gm(\xi,D_x\xi)),\\
D_t^2\xi&=\pd_t(D_t\xi)+\Gm(\eta,D_t\xi)
=\xi_{tt}+\pd_t(\Gm(\eta,\xi))+\Gm(\eta,D_t\xi)),
\end{aligned}
\end{equation}
the equation
$D_t^2\xi-D_x^2\xi=(\nr{D_x\xi}^2-\nr{D_t\xi}^2)\xi+\te^\perp$
is rewritten as
\begin{equation}
\begin{aligned}
\xi_{tt}-\xi_{xx}
&=-\pd_t(\Gm(\eta,\xi))-\Gm(\eta,D_t\xi)
+\pd_x(\Gm(\xi,\xi))+\Gm(\xi,D_x\xi)\\
&\qquad+(\nr{D_x\xi}^2-\nr{D_t\xi}^2)\xi
+\te^\perp.
\end{aligned}
\end{equation}
We express the right hand side using
\begin{equation}
\label{eq:F-H-byGmxietate}
\begin{aligned}
&F(x,t,\xi,\xi_x,\xi_t)
=-\pd_t(\Gm(\eta,\xi))-\Gm(\eta,D_t\xi)+\Gm(\xi,D_x\xi)\\
&\qquad\qquad\qquad\qquad+(\nr{D_x\xi}^2-\nr{D_t\xi}^2)\xi
+\te-g(\te,\xi)\xi,\\
&H(x,t,\xi)=\Gm(\xi,\xi),
\end{aligned}
\end{equation}
and apply the above lemmas.
The functions $F$ and $H$ are polynomials in $\xi$, $\xi_x$ and $\xi_t$,
and their coefficients contain only $\Gm$, $\Gm_t$, $\eta$, $\eta_t$ and $\te$.
And, $\Gm=\Gm(\gm)$ has same differentiability to $\gm$.
Therefore,
the function $F$ satisfies the condition in Lemma \ref{lemma:wave-eq-C1-existence}, when $\gm$, $\gm_t$, $\eta$, $\eta_t$, $\te\in C^0$.

\begin{remark}
If we use $\Dx$ instead of $D_x$, then the expression of $f$ contains $(\Gm(\gm_x,\xi))_x$, and coefficients are not continuous unless $\gm\in C^2$.
\end{remark}

We make sure that the solution $\xi=u\in C^1$ satisfies $\nr\xi^2\equiv1$.
Since it is not of class $C^2$, we need careful consideration.

\begin{proposition}
\label{proposition:nrxi2=1}
If the initial data satisfies $\nr{a}^2=1$ and $g(a,b)=0$,
then $\nr\xi^2=1$ identically.
\end{proposition}

\begin{proof}
Under notation of Lemma \ref{lemma:I(a,b,f,h)-C1-formula},
\begin{equation}
\label{eq:vpm}
\begin{aligned}
u^\pm(x\mp t,t)&=a'(x)\pm b(x)
\pm\int_0^tf(x\mp\tau,\tau)\,d\tau\\
&\qquad+\{h(x,0)-h(x\mp t,t)\}
+\int_0^th_t(x\mp\tau,\tau)\,d\tau.
\end{aligned}
\end{equation}
Hence,
\begin{equation}
\label{eq:upm+h}
(u^\pm+h)(x\mp t,t)
=a'(x)\pm b(x)+h(x,0)+\int_0^t(\pm f+h_t)(x\mp\tau,\tau)\,d\tau.
\end{equation}

The left hand side of \eqref{eq:upm+h} becomes, using $k=\Gm(\eta,u)$,
\begin{equation}
u^\pm+h
=u_x\pm u_t+\Gm(u,u)
=D_xu\pm D_tu\mp k.
\end{equation}
And, the right hand side of \eqref{eq:upm+h} is differentiable with respect to $t$.
Hence,
\begin{equation}
\pd_t\{(D_xu\pm D_tu\mp k)(x\mp t,t)\}
=(\pm f+h_t)(x\mp t,t).
\end{equation}

Therefore, along the line $x=x_0\mp\tau$ and $t=\tau$, it holds that
\begin{equation}
\label{eq:g(dds,u)}
g(\dds(\pm D_xu+D_tu-k),u)=g(f\pm h_t,u).
\end{equation}

The left hand side of \eqref{eq:g(dds,u)} becomes, remarking that $k$ is orthogonal to $u$,
\begin{equation}
\begin{aligned}
&\dds g(\pm D_xu+D_tu,u)-g(\pm D_xu+D_tu-k,\dds u)\\
&\qquad=\frac12\dds\{\pm\pd_x(\nr u^2)+\pd_t(\nr u^2)\}
-g(\pm D_xu+D_tu-k,u_t\mp u_x).
\end{aligned}
\end{equation}
Here,
\begin{equation}
\begin{aligned}
&-g(\pm D_xu+D_tu-k,u_t\mp u_x)
=g(\mp D_xu-D_tu+k,D_tu\mp D_xu-k\pm h)\\
&\qquad=\nr{D_xu}^2-\nr{D_tu}^2
-g(D_xu,h)+g(D_tu,2k\mp h)+g(k,-k\pm h).\\
\end{aligned}
\end{equation}

We substitute
$f=-k_t-\Gm(\eta,D_tu)+\Gm(u,D_xu)+(\nr{D_xu}^2-\nr{D_tu}^2)u+\te-g(\te,u)u$
into the right hand side of \eqref{eq:g(dds,u)}.
By the anti-symmetry of $\Gm$ and the orthogonality of $h$ and $k$ to $u$,
we have
\begin{equation}
\begin{aligned}
&g(-k_t-\Gm(\eta,D_tu)+\Gm(u,D_xu)\pm h_t,u)\\
&\qquad=g(k\mp h,u_t)+g(\Gm(\eta,u),D_tu)-g(\Gm(u,u),D_xu)\\
&\qquad=g(k\mp h,D_tu-k)+g(k,D_tu)-g(h,D_xu)\\
&\qquad=g(D_tu,2k\mp h)-g(D_xu,h)
+g(k,-k\pm h).
\end{aligned}
\end{equation}

And,
\begin{equation}
\begin{aligned}
&g((\nr{D_xu}^2-\nr{D_tu}^2)u+\te-g(\te,u)u,u)\\
&\qquad=\nr u^2(\nr{D_xu}^2-\nr{D_tu}^2)+(1-\nr u^2)g(\te,u).
\end{aligned}
\end{equation}

Combining them, we get
\begin{equation}
\dds\{\pm\pd_x(\nr u^2)+\pd_t(\nr u^2)\}
=2(\nr u^2-1)(\nr{D_xu}^2-\nr{D_tu}^2-g(\te,u)).
\end{equation}
Let $\al$ be the right hand side of this equation.
By integrating this equation, we have
\begin{equation}
\begin{aligned}
&\pm\pd_x(\nr u^2)+\pd_t(\nr u^2)
=2\int_0^t\al(x\pm(t-\tau),\tau)\,d\tau,\\
&\pd_t(\nr u^2)
=\int_0^t\al(x+(t-\tau),\tau)+\al(x-(t-\tau),\tau)\,d\tau.
\end{aligned}
\end{equation}
Since $u$ is $C^1$-bounded, we get
$m_0(\al)(t)\le Ap(t)$,
with $p(t):=m_0(\nr u^2-1)$ and a positive constant $A$.
Since the initial value satisfies $\nr u^2-1=\pd_t(\nr u^2-1)=0$,
when $t\le T$,
\begin{equation}
\begin{aligned}
&\nr{\pd_t(\nr u^2-1)}\le2A\int_0^tp(\tau)\,d\tau,\\
&\nr{\nr u^2-1}\le2A\int_0^t\int_0^sp(\tau)\,d\tau ds
=2A\int_0^t\int_\tau^tp(\tau)\,dsd\tau\\
&\qquad=2A\int_0^t(t-\tau)p(\tau)\,d\tau
\le2AT\int_0^tp(\tau)\,d\tau.
\end{aligned}
\end{equation}

Therefore,
\begin{equation}
p(t)\le2AT\int_0^tp(\tau)\,d\tau.
\end{equation}

By solving this integral inequality, we see that $p(t)=0$ and $\nr u^2-1=0$ identically.
\end{proof}

\begin{proposition}
\label{proposition:wave-eq-C1-solution}
We consider the semilinear integral equation \eqref{eq:wave-integral-equation}.
We assume that $\Gm\in\Cls01$, $\eta\in\Cls01$, $\te\in C^0$ and
$\Gm(x,0)\in C^1(x)$,
and that the initial values $\xi(x,0)=a(x)\in C^1(x)$ and $\xi_t(x,0)=b(x)\in C^0(x)$ satisfy
$\nr a^2=1$ and $g(a,b)=0$.
We also assume that, for $T_0>0$ and $r_0>2m_0(a,a',b)$,
\begin{equation}
\max_{S(T_0,r_0)}\{\nr\Gm,\nr{\Gm_t},\nr\eta,\nr{\eta_t},\nr\te\}\le K_0,
\quad\max_{S(0,r_0)}\{\nr{\Gm_x}\}\le K_0.
\end{equation}
Then, there exists a positive constant $T_1\le T_0$ depending only on $m_0(a,a',b)$, $r_0$, $T_0$ and $K_0$,
such that there exists a unique solution $\xi(x,t)\in C^1$ on the time interval $[0,T_1]$,
and the solution satisfies $M_1(\xi,T_1)\le r_0$ and $\nr\xi^2=1$.
\end{proposition}

\begin{proof}
The function $F$ is a polynomial in $\xi$, $\xi_x$, $\xi_t$, $\Gm$, $\Gm_t$, $\eta$, $\eta_t$ and $\te$,
and the function $H$ is a polynomial in $\xi$ and $\Gm$.
And, $H_x(x,0,u)=\Gm_x(x,0)(u,u)$.
Hence, by the assumption, there is a positive constant $K_1$ depending only on $r_0$ and $K_0$, such that
\begin{equation}
\begin{aligned}
&\max_{S(T_0,r_0)}\{\nr F,\nr{F_u},\nr{F_p},\nr{F_q},\nr H,\nr{H_t},\nr{H_u},\nr{H_{tu}},\nr{H_{uu}}\}\le K_1,\\
&\max_{S(0,r_0)}\{\nr{H_x}\}\le K_1.
\end{aligned}
\end{equation}
I.e., the condition of Lemma \ref{lemma:wave-eq-M1-existence-uHC1} is satisfied.
Therefore, there exists a positive constant $T_1\le T_0$ depending only on $m_0(a,a',b)$, $T_0$, $r_0$ and $K_1$, such that there is a unique solution $\xi\in C^1$ on the time interval $[0,T_1]$.
Moreover, the solution $\xi$ satisfies $\nr\xi^2=1$ by Lemma \ref{proposition:nrxi2=1}.
\end{proof}

\begin{lemma}
\label{lemma:wave-eq-C1-solution-difference}
We assume that, for $i=1,2$, functions $\Gm_i$, $\eta_i$, $\te_i\in C^0$ and common initial values $a\in C^1(x)$ and $b\in C^0(x)$ satisfy
the assumption of Proposition \ref{proposition:wave-eq-C1-solution}.
Namely, it holds that $(\Gm_i),(\eta_i)\in\Cls01$,
and $T_0>0$ and $r_0>2m_0(a,a',b)$ are given, and
\begin{equation}
\begin{aligned}
&\max_{S(T_0,r_0)}\{\nr{\Gm_i},\nr{(\Gm_i)_t},\nr{\eta_i},\nr{(\eta_i)_t},\nr{\te_i}\}\le K_0,
\quad\max_{S(0,r_0)}\{\nr{(\Gm_i)_x}\}\le K_0,\\
&\nr a^2=1,
\quad g(a,b)=0.
\end{aligned}
\end{equation}
We also assume that $\Gm_1(x,0)=\Gm_2(x,0)$,
and, there is a solution $\xi_i$ on a time interval $[0,T]\subset[0,T_0]$
satisfying $M_1(\xi_i,T)\le r_0$ for each $i$,
and that $\max_{S(T,r_0)}\{\nr{\dl\Gm},\nr{\dl\Gm_t},\nr{\dl\eta},\nr{\dl\eta_t},\nr{\dl\te}\}\le\ep$, where $\dl p:=p_2-p_1$.

Then, there exists a positive constant $K_1$ depending only on $T$, $r_0$ and $K_0$
such that $M_1(\dl\xi,t)\le K_1t\ep$ for $t\in[0,T]$.
\end{lemma}

\begin{proof}
Since the function $F$ is a polynomial in $\xi$, $\xi_x$, $\xi_t$, $\Gm$, $\Gm_t$, $\eta$, $\eta_t$ and $\te$,
and the function $H$ is a polynomial in $\xi$ and $\Gm$,
there is a positive constant $K_1$ depending only on $r_0$ and $K_0$
such that $\max_{S(T,r_0)}\{\nr{\dl F},\nr{\dl H},\nr{\dl H_t},\nr{\dl H_u}\}\le K_1\ep$ holds.
Namely, in Lemma \ref{lemma:wave-eq-C1-difference}, the assumption holds for $\ep_1=K_1\ep$ and $\ep_2=0$.
Therefore there exists a positive constant $K_2$ depending only on $T$, $r_0$ and $K_0$, such that $M_1(\dl\xi,t)\le K_2t\ep$ holds.
\end{proof}

\subsection
{Estimation of curve bentness}
\label{subsection:Bgmxi}

We study the curve bentness $B$ to analyze equation $\rm(O_\te)$.
The curve bentness is defined as follows.
\begin{equation}
B(\gm,\xi):=\inf_\phi\{\Nr{\phi-\xi}^2+\Nr{D_x\phi}^2\}^{1/2},
\quad B(\gm):=B(\gm,\gm').
\end{equation}
Here, we assume that $\nr{\xi}^2=1$.
We will see that $B(\gm)>0$ means $\gm$ is not a geodesic.
Note that $\Nr{\xi}=1$, because curves are mapping $\gm:S^1={\bf R}/{\bf Z}\to(M,g)$ in this paper.

\begin{lemma}
For all $\gm$ and $\xi$, it holds that $B(\gm,\xi)\le1$.
And, the infimum $B(\gm,\xi)$ is attained by $\phi$ satisfying $-D_x^2\phi+\phi=\xi$.
\end{lemma}

\begin{proof}
Let $B(\gm,\xi,\phi):=(\Nr{\phi-\xi}^2+\Nr{D_x\phi}^2)^{1/2}$.
First, it holds that
\begin{equation}
B(\gm,\xi)\le B(\gm,\xi,0)
=\Nr\xi
=1.
\end{equation}

Next, we show existence of $\phi\in H^1$ satisfying $-D_x^2\phi+\phi=\xi$.
The operator $L(\phi):=-D_x^2\phi+\phi$ is self-adjoint.
And, if $-D_x^2\phi+\phi=0$, then
\begin{equation}
0
=\Ip{-D_x^2\phi+\phi,\phi}
=\Ip{D_x\phi,D_x\phi}+\Ip{\phi,\phi}
\end{equation}
and $\phi=0$.
Hence, $\Ker L=0$.
Therefore, the equation $L(\phi)=\xi$ has a solution.

Let $\phi$ be the solution, and choose any $\td\phi=\phi+\dl$.
Then, in the equality
\begin{equation}
\begin{aligned}
&B(\gm,\xi,\phi+\dl)^2-B(\gm,\xi,\phi)^2
=\Nr{\dl}^2+2\Ip{\dl,\phi-\xi}
	+\Nr{D_x\dl}^2+2\Ip{D_x\dl,D_x\phi},
\end{aligned}
\end{equation}
we have
\begin{equation}
\Ip{\dl,\phi-\xi}+\Ip{D_x\dl,D_x\phi}
=\Ip{\dl,\phi-\xi}-\Ip{\dl,D_x^2\phi}
=\Ip{\dl,-D_x^2\phi+\phi-\xi}
=0,
\end{equation}
hence
\begin{equation}
B(\gm,\xi,\phi+\dl)^2-B(\gm,\xi,\phi)^2
=\Nr{\dl}^2+\Nr{D_x\dl}^2
\ge0.
\end{equation}
Therefore, $\phi$ attains the infimum $B(\gm,\xi)$.
\end{proof}

\begin{lemma}
$B(\gm,\xi)=0$ if and only if $D_x\xi=0$.
\end{lemma}

\begin{proof}
A function $\phi$ attaining $B(\gm,\xi)=0$ satisfies $\phi=\xi$ and $D_x\phi=0$.
Therefore $D_x\xi=0$.
\end{proof}

And, the functional $B$ is continuous in the following sense.

\begin{lemma}
\label{lemma:B(gm,xi)-estimate}
The functional $B(\gm,\xi)$ is continuous with respect to $L_2(x)$-topology of $\Gm$ and $\xi$.
\end{lemma}

\begin{proof}
To show continuity in $\xi$, consider $\gm$, $\xi_1$ and $\xi_2$,
and let $\dl\xi=\xi_2-\xi_1$.
Take $\phi$ attaining $B(\gm,\xi_1)$.
Then, we have $\Nr{\phi-\xi_1}\le B(\gm,\xi_1)$.
Therefore,
\begin{equation}
\begin{aligned}
&B(\gm,\xi_2)^2-B(\gm,\xi_1)^2
\le B(\gm,\xi_2,\phi)^2-B(\gm,\xi_1,\phi)^2
=\Nr{\phi-\xi_2}^2-\Nr{\phi-\xi_1}^2\\
&\qquad=\Nr{\dl\xi}^2+2\Ip{\phi-\xi_1,\dl\xi}
\le\Nr{\dl\xi}^2+2\Nr{\phi-\xi_1}\Nr{\dl\xi}
\le\Nr{\dl\xi}^2+2\Nr{\dl\xi}B(\gm,\xi_1).
\end{aligned}
\end{equation}
I.e., $B(\gm,\xi_2)^2\le(B(\gm,\xi_1)+\Nr{\dl\xi})^2$.
Since the inequality exchanging $\xi_1$ and $\xi_2$ similarly holds, we get
\begin{equation}
\nr{B(\gm,\xi_1)-B(\gm,\xi_2)}\le\Nr{\xi_1-\xi_2}.
\end{equation}

Next, to show continuity in $\Gm$, consider $\Gm_1$, $\Gm_2$ and $\xi$.
We use corresponding notations $\gm_i$ and $D_i$.
And, let $\dl\Gm=\Gm_2-\Gm_1$.
If $\phi$ attains $B(\gm_1,\xi)$, then $\Nr{\phi-\xi}$, $\Nr{(D_1)_x\phi}\le B(\gm_1,\xi)\le1$, and
\begin{equation}
\begin{aligned}
&\Nr{\phi}\le\Nr{\phi-\xi}+\Nr\xi
\le1+1
=2,\\
&m_0(\phi)
\le\Nr{\phi}+\Nr{(D_1)_x\phi}
\le2+1
=3.
\end{aligned}
\end{equation}
Hence,
\begin{equation}
\begin{aligned}
&B(\gm_2,\xi)^2\le B(\gm_2,\xi,\phi)^2
=\Nr{\phi-\xi}^2+\Nr{(D_1)_x\phi+\dl\Gm(\xi,\phi)}^2\\
&\qquad=B(\gm_1,\xi,\phi)^2+\Nr{\dl\Gm(\xi,\phi)}^2+2\Ip{(D_1)_x\phi,\dl\Gm(\xi,\phi)}\\
&\qquad\le B(\gm_1,\xi)^2+\Nr{\dl\Gm(\xi,\phi)}^2+2B(\gm_1,\xi)\Nr{\dl\Gm(\xi,\phi)}
=(B(\gm_1,\xi)+\Nr{\dl\Gm(\xi,\phi)})^2.
\end{aligned}
\end{equation}

Therefore,
\begin{equation}
\begin{aligned}
&B(\gm_2,\xi)-B(\gm_1,\xi)\le\Nr{\dl\Gm(\xi,\phi)}
\le\Nr{\dl\Gm}m_0(\phi)
\le3\Nr{\dl\Gm}.
\end{aligned}
\end{equation}
Similarly, we have $B(\gm_1,\xi)-B(\gm_2,\xi)\le3\Nr{\dl\Gm}$, and
\begin{equation}
\nr{B(\gm_1,\xi)-B(\gm_2,\xi)}
\le3\Nr{\Gm_1-\Gm_2}.
\end{equation}
\end{proof}

\subsection
{Estimation for Equation $\rm(O_{\protect\te})$}
\label{subsection:O-te}

In this section, we estimate solutions to Equation $\rm(O_\te)$
$-D_x^2\te+\te^\perp=D_x\Psi+\Phi$.
More precisely, for known functions $\xi$, $\Psi$, $\Phi\in C^0$, we estimate the unknown function $\te\in\Cls10$.
This equation is a linear ODE, but we need geometric consideration because it reflects the condition that the curve is away from geodesics.
In fact, if $\gm(x)$ is a geodesic and $\xi=\gm_x$, then
$-D_x^2\xi+\xi^\perp=0$
and so the uniqueness of $\te$ doesn't hold
and $-D_x^2\te+\te^\perp=\xi$ doesn't have solutions.
First we analyze this equation with arbitrarily fixed $t$.

Although we assume $\Gm$ is only class $C^0(x)$, the equation $-D_x^2u+u^\bot=h$ is converted to the equation $-\td u_{xx}+\td u^\bot=\td h$ (with non-trivial holonomy) using a frame field parallel with respect to $D_x$ along $\gm$.
Therefore, we need not care about differentiability of $\Gm$.

\begin{lemma}
\label{lemma:O-te-base}
Consider equation $-D_x(D_xu+f)+u^\perp=h$ for $u$ with known functions $f$, $h\in L_2(x)$.
If $B(\gm,\xi)\ne0$, then there is a unique solution $u\in H^1$.
Moreover, it holds that
\begin{equation}
m_0(u),m_0(D_xu)\le C(\Nr{f}+\Nr{h}),
\end{equation}
where $C$ is a positive constant monotone increasingly depending only on $\Nr{D_x\xi}$ and $B(\gm,\xi)^{-1}$.
\end{lemma}

\begin{proof}
First we estimate under assumption $f\in H^1$,
and remove the assumption at last step.
Since $\Ip{-D_x^2u+u^\perp,p}=\Ip{u,-D_x^2p+p^\perp}$
and $\Ip{D_x^2u+u^\perp,u}=\Nr{D_xu}^2+\Nr{u^\perp}^2\ge0$,
the operator $L(u)=D_x^2u+u^\perp$ is self-adjoint and semi-positive.
Elements $u$ of $\Ker L$ satisfy $D_xu=0$ and $u^\perp=0$.
Hence, we get $u=v\xi$ and $0=D_x(v\xi)=v_x\xi+vD_x\xi$.
Since $D_x\xi$ is orthogonal to $\xi$, we get $v_x=0$ and $D_x\xi=0$.
Hence, if $u\ne0$, then $B(\gm,\xi)\le B(\gm,\xi,\xi)=0$.
I.e., $\Ker L=0$ under assumption $B(\gm,\xi)>0$.
Therefore, equation $L(u)=D_xf+h\in L_2(x)$ has a unique solution $u\in H^1(x)$.

To simplify notations, we denote by small $c_*$ positive constants monotone increasingly depending only on $\Nr{D_x\xi}$ but independent of $B(\gm,\xi)$.
To show boundedness of $u$, let $\al:=\Nr f+\Nr h$ and contrarily assume that $\Nr{u}\ge\al$ for a moment.
Then, $\al^\sigma\Nr u^\tau\le\al^{\sigma-\rho}\Nr u^{\tau+\rho}$
for $\rho\ge0$.

From
\begin{equation}
\begin{aligned}
\label{eq:Dxte+1/2p}
&\Nr{D_xu+\frac12f}^2+\Nr{u^\perp}^2
=\Ip{-D_x^2u-D_xf+u^\perp,u}+\frac14\Nr{f}^2\\
&\qquad=\Ip{h,u}+\frac14\Nr{f}^2
\le\Nr h\Nr u+\frac14\Nr{f}^2
\le\al\Nr u,
\end{aligned}
\end{equation}
we have $\Nr{D_xu+(1/2)f}^2$, $\Nr{u^\perp}^2\le\al\Nr u$.

Let $v:=g(u,\xi)$.
Then $u=u^\perp+v\xi$,
$\Nr u^2=\Nr{u^\perp}^2+\Nr v^2$,
and,
\begin{equation}
\begin{aligned}
&\Nr{D_xu+\frac12f}^2,\Nr{u^\perp}^2
\le\al\Nr u,\\
&\Nr{D_xu}
\le\al^{1/2}\Nr u^{1/2}+\Nr f
\le2\al^{1/2}\Nr u^{1/2}
\le3\Nr u,\\
&m_0(u^\perp)\le m_0(u)
\le\Nr u+\Nr{D_xu}
\le4\Nr u.
\end{aligned}
\end{equation}
From
\begin{equation}
g(D_xu,\xi)=g(D_xu^\perp+D_x(v\xi),\xi)
=-g(u^\perp,D_x\xi)+v_x,
\end{equation}
we have
\begin{equation}
\begin{aligned}
&\Nr{v_x}\le\Nr{D_xu}+m_0(u^\perp)\Nr{D_x\xi}
\le3\Nr u+\C[c]\Nr u
\le\C[c]\Nr u,\\
&m_0(v)\le\Nr v+\Nr{v_x}\le\C[c]\Nr u,\\
&\Nr{D_x(u^\perp)}\le\Nr{D_xu}+\Nr{D_x(v\xi)}
\le3\Nr u+\Nr{v_x}+\C[c]m_0(v)
\le\C[c]\Nr u.
\end{aligned}
\end{equation}
Therefore,
\begin{equation}
\begin{aligned}
m_0(u^\perp)^2&\le\Nr{u^\perp}(\Nr{u^\perp}+\Nr{D_xu^\perp})
\le\C[c]\al^{1/2}\Nr u^{1/2}(\al^{1/2}\Nr u^{1/2}+\Nr u)\\
&\le\C[c]\al^{1/2}\Nr u^{3/2}.
\end{aligned}
\end{equation}

Using this new estimation of $m_0(u^\perp)$, we get
\begin{equation}
\Nr{v_x}\le\Nr{D_xu}+\C[c]m_0(u^\perp)
\le2\al^{1/2}\Nr u^{1/2}+\C[c]\al^{1/4}\Nr u^{3/4}
\le\Cl{c:vx}\al^{1/4}\Nr u^{3/4}.
\end{equation}

Now we choose $x_0$ such that $\Nr{v}=\nr{v(x_0)}\le\Nr u$,
and let $b:=\Nr u$ or $b:=-\Nr u$ according to $v(x_0)\ge0$ or $v(x_0)<0$.
Then, $\nr{b-v(x_0)}=\Nr u-\Nr v$, and
\begin{equation}
\begin{aligned}
&m_0(b-v)\le\nr{b-v(x_0)}+m_0(v-v(x_0))
\le\Nr u-\Nr v+\Nr{v_x}\\
&\qquad=(\Nr u^2-\Nr v^2)/(\Nr u+\Nr v)+\Nr{v_x}
\le\Nr{u^\perp}^2/\Nr u+\Nr{v_x}\\
&\qquad\le\al\Nr u/\Nr u+\Cr{c:vx}\al^{1/4}\Nr u^{3/4}
\le\C[c]\al^{1/4}\Nr u^{3/4}.
\end{aligned}
\end{equation}

Using this, we get
\begin{equation}
\begin{aligned}
&\Nr{D_x(u^\perp+b\xi)}\le\Nr{D_x(u^\perp+v\xi)}+\Nr{D_x((b-v)\xi)}\\
&\qquad\le\Nr{D_xu}+\Nr{v_x}+m_0(b-v)\Nr{D_x\xi}\\
&\qquad\le2\al^{1/2}\Nr u^{1/2}+\C[c]\al^{1/4}\Nr u^{3/4}
+\C[c]\al^{1/4}\Nr u^{3/4}
\le\C[c]\al^{1/4}\Nr u^{3/4},
\end{aligned}
\end{equation}
and
\begin{equation}
\begin{aligned}
&B(\gm,\xi)^2\le B(\gm,\xi,b^{-1}u^\perp+\xi)^2
=b^{-2}\Nr{u^\perp}^2+b^{-2}\Nr{D_x(b\xi+u^\perp)}^2\\
&\qquad\le\C[c]\Nr u^{-2}\{\al\Nr u+\al^{1/2}\Nr u^{3/2}\}
\le\C[c]\al^{1/2}\Nr u^{-1/2}.
\end{aligned}
\end{equation}
Therefore $\Nr u\le\Cl{c:u<al}B^{-4}\al$.

It means that either $\Nr u\le\al$ or $\Nr u\le\Cr{c:u<al}B^{-4}\al$ holds.
Thus, finally we get
\begin{equation}
\Nr u\le\Cl{c:u<B-4al}\max\{\al,B^{-4}\al\}
=\Cr{c:u<B-4al}B^{-4}(\Nr f+\Nr h),
\end{equation}
which is the desired estimation of $\Nr u$.

Next, we estimate the differential of $u$.
We denote by $K_*$ positive constants depending only on $\Nr{D_x\xi}$ and $B(\gm,\xi)^{-1}$ opposed to $c_*$.
From \eqref{eq:Dxte+1/2p} and the above estimation of $\Nr u$,
\begin{equation}
\Nr{D_xu+\frac12f},\Nr{u^\perp}\le\al^{1/2}\Nr u^{1/2}
\le\Cl[K]{K:uperp<al}\al.
\end{equation}

Hence,
\begin{equation}
\begin{aligned}
&\Nr{D_xu}\le3\Nr u
\le\Cl[K]{K:Dxu}\al,\\
&\Nr{D_xu+f}\le\Cr{K:Dxu}\al+\Nr f
\le\C[K]\al,\\
&\Nr{D_x(D_xu+f)}=\Nr{h-u^\perp}
\le\Nr h+\Cr{K:uperp<al}\al
\le\C[K]\al.
\end{aligned}
\end{equation}

From these inequalities, we have
\begin{equation}
\begin{aligned}
&m_0(u)\le\Nr u+\Nr{D_xu}
\le\C[K]\al,\\
&m_0(D_xu+f)\le\Nr{D_xu+f}+\Nr{D_x(D_xu+f)}
\le\C[K]\al.
\end{aligned}
\end{equation}

In the above, we assumed $f\in H^1$.
However the resulting inequalities doesn't contain $\Nr{D_xf}$ and contains only $\Nr{f}$.
Therefore, when we take $L_2$ approximation of $f\in L_2$ by $H_1$ functions, the solution converges with respect to $H^1$ topology.
Namely, the same conclusion holds even when we only assume $f\in L_2$.
\end{proof}

We simplify Lemma \ref{lemma:O-te-base} in terms of $m_0$ norm.

\begin{lemma}
\label{lemma:O-te-base-C1}
Let $f$, $h\in C^0(x)$ be known functions,
and consider equation $-D_x(D_xu+f)+u^\perp=h$ for $u$ assuming $B(\gm,\xi)>0$.
Then, there is a unique solution $u\in C^1(x)$.
Moreover, it holds that
\begin{equation}
m_0(u),m_0(D_xu)
\le C(m_0(f)+m_0(h)),
\end{equation}
where $C$ is a positive constant depending only on $\Nr{D_x\xi}$ and $B(\gm,\xi)^{-1}$, and the dependence is monotone increasing.
\end{lemma}

\begin{proof}
The estimation of $m_0$ norm directly follows from Lemma \ref{lemma:O-te-base}.
For the continuity, we see
\begin{equation}
\begin{aligned}
&\nr{u(x+\dl)-u(x)}=\nr{\int_x^{x+\dl}u_x\,dx}
=\nr{\int_x^{x+\dl}D_xu-\Gm(\xi,u)\,dx}\\
&\qquad\le\dl\{m_0(D_xu)+m_0(\Gm)m_0(u)\},\\
\end{aligned}
\end{equation}
\begin{equation}
\begin{aligned}
&\nr{(D_xu+f)(x+\dl)-(D_xu+f)(x)}
=\nr{\int_x^{x+\dl}\pd_x(D_xu+f)\,dx}\\
&\qquad=\nr{\int_x^{x+\dl}D_x(D_xu+f)-\Gm(\xi,D_xu+f)\,dx}\\
&\qquad=\nr{\int_x^{x+\dl}h-u^\perp-\Gm(\xi,D_xu+f)\,dx}\\
&\qquad\le\dl\{m_0(h)+m_0(u)+m_0(\Gm)m_0(D_xu+f)\}.
\end{aligned}
\end{equation}
Therefore, $u$, $D_xu\in C^0(x)$.
\end{proof}

Next, we estimate influence of $\Gm$, $\xi$, $f$ and $h$ in the equation $-D_x(D_xu+f)+u^\perp=h$.
We use same notation as Proof of Lemma \ref{lemma:B(gm,xi)-estimate}.

\begin{lemma}
\label{lemma:O-te-base-difference}
Assume that, for indices $i=1,2$, functions $\Gm_i$, $\xi_i$, $f_i$ and $h_i$ are given and satisfy
$m_0(\Gm_i)$, $m_0(f_i)$, $m_0(h_i)\le K$.
We also assume that $\Nr{D_x\xi}$, $B^{-1}\le K$ for each i.
Let $u_i$ be the solution to the equation $-(D_i)_x((D_i)_xu+f_i)+u-g(u,\xi_i)u=h_i$.

Let $\dl*=*_2-*_1$.
Then, there is a positive constant $K_1$ depending only on $K$, such that $m_0(\dl u)$, $m_0(\dl u_x)\le K_1\ep$ if $m_0(\dl\Gm)$, $m_0(\dl\xi)$, $m_0(\dl f)$, $m_0(\dl h)\le\ep$.
\end{lemma}

\begin{proof}
We take difference of $2$ equations and get
\begin{equation}
\label{eq:diff-eq-te}
\begin{aligned}
&-(D_2)_x((D_2)_x\dl u+(\dl\Gm)(\xi_2,u_1)+\Gm_1(\dl\xi,u_1)+\dl f)
+\dl u-g(\dl u,\xi_2)\xi_2\\
&\qquad=(\dl\Gm)(\xi_2,(D_1)_xu_1+f_1)+\Gm_1(\dl\xi,(D_1)_xu_1+f_1)\\
&\qquad\qquad\qquad+g(u_1,\dl\xi)\xi_2+g(u_1,\xi_1)\dl\xi+\dl h.
\end{aligned}
\end{equation}

We apply to this equation Lemma \ref{lemma:O-te-base-C1} with $u=\dl u$, $D=D_2$, $\xi=\xi_2$,
$f=(\dl\Gm)(\xi_2,u_1)+\Gm_1(\dl\xi,u_1)+\dl f$
and $h=\text{(right hand side of \eqref{eq:diff-eq-te})}$.
Since also $m_0((D_1)_xu_1)$ is bounded by a positive constant $K_1$ depending only on $K$,
there is a positive constant $K_2$ depending only on $K$ such that
$m_0(f)$, $m_0(h)\le K_2\ep$.

Therefore, there is a positive constant $K_3$ depending only on $K$ such that
$m_0(\dl u)$, $m_0((D_2)_x\dl u)\le K_3\ep$.
Since $(D_2)_x\dl u=\dl u_x+\Gm_2(\xi_2,\dl u)$,
we get $m_0(\dl u_x)\le K_3\ep+KK_3\ep$.
\end{proof}

\begin{corollary}
\label{corollary:O-te-t-dependence}
In the equation $-D_x(D_xu+f)+u^\perp=h$ with $B>0$, if $\Gm=\Gm(x,t)$, $f=f(x,t)$ and $h=h(x,t)$
are of class $C^0$,
then the solution $u=u(x,t)$ satisfies $u$, $u_x$, $D_x(D_xu+f)\in C^0$.
\end{corollary}

We apply the above to $\rm(O_\te)$.

\begin{proposition}
\label{proposition:O-te-sol-and-diff}
The linear ODE $\rm(O_\te):$ $-D_x(D_x\te+\Psi)+\te^\perp=\Phi$ for $\te$ with $B>0$ has a unique continuous solution $\te(x,t)\in\Cls10$ if $\Gm(x,t)$, $R(x,t)\in C^0$, $\xi(x,t)\in C^1$ and $\eta(x,t)\in C^0$.
The solution satisfies $m_0(\te)$, $m_0(\te_x)\le K_1$.
Here, the domain of $t$ is $[0,T]$, and $K_1$ is a positive constant depending only on $M_0(\Gm,T)$, $M_0(R,T)$, $M_1(\xi,T)$, $M_0(\eta,T)$ and $B^{-1}$.

Moreover, the difference of solutions for two data are estimated as follows.
Let functions $\Gm_i$, $R_i$, $\xi_i$ and $\eta_i$ satisfy the above condition for $i=1,2$.
Let $\dl*=*_2-*_1$.
Then there is a positive constant $K_2$ depending only on $M_0(\Gm,T)$, $M_0(R,T)$, $M_1(\xi,T)$, $M_0(\eta,T)$ and $B^{-1}$,
such that $m_0(\dl\te)$, $m_0(\dl\te_x)\le K_2\ep$
if $m_0(\dl\Gm)$, $m_0(\dl R)$, $m_1(\dl\xi)$, $m_0(\dl\eta)\le\ep$.
\end{proposition}

\begin{proof}
The functions $\Psi=R(\xi,D_x\xi)\xi-R(\xi,D_t\xi)\eta$
and $\Phi=(\nr{D_t\xi}^2-\nr{D_x\xi}^2)\xi-R(\xi,\eta)\eta$
are of class $C^0$, and satisfy $m_0(\Psi)$, $m_0(\Phi)\le K_1$.
Therefore, by applying Lemma \ref{lemma:O-te-base-C1},
we see $m_0(\te)$, $m_0(D_x\te)\le CK_1$.
And, by Lemma \ref{lemma:O-te-base-difference}, we get $\te\in\Cls10$.

In the latter half, there is a positive constant $K_2$ depending only on $K_1$ such that
$m_0(\dl\Psi)$, $m_0(\dl\Phi)\le K_2\ep$,
and there is another positive constant $K_3$ such that $m_0(\dl\te)$, $m_0(\dl\te_x)\le K_3\ep$ by Lemma \ref{lemma:O-te-base-difference}.
\end{proof}

\subsection
{Existence of a short time solution}
\label{subsection:short-time}

Let $\gm(x,0)=a(x)\in C^1(x)$, $\eta(x,0)=b(x)\in C^0(x)$, $\xi(x,0)=\td a(x)\in C^1(x)$ and $\xi_t(x,0)=\td b(x)\in C^0(x)$ be the initial data.
We assume that $\nr{\td a}^2=1$, $g(\td a,\td b)=0$ and $B(a,\td a)>0$.

We also asuume that the curve $a(x)$ is regular, so that there is a $C^\infty$ orthonormal frame field $\{e_i\}$ on the whole tubular neibourhood.
Note that we may assume that the tubular neibourhood is topologically tirivial by taking a double covering if necessary.
We use differentiability of $a(x)$ only for the existence of the frame field.
Except that, we need only continuousness of $a(x)$.
However, the initial curve $a(x)$ is a unit speed $C^1$ curve in the equation of motion \eqref{eq:base-single} anyway.

We define the space of functions on the time interval $[0,T]$ satisfying the initial condition, as follows.
\begin{equation}
\begin{aligned}
&\mathcal S_\gm(T,r)
:=\{\gm\in\Cls01\mid
M_{0,1}(\gm,T)\le r,
\gm^i(x,0)=a^i(x)\},\\
&\mathcal S_\eta(T,r)
:=\{\eta\in\Cls01\mid
M_{0,1}(\eta,T)\le r,
\eta^i(x,0)=b^i(x)\},\\
&\mathcal S_\xi(T,r)
:=\{\xi\in C^1\mid
M_1(\xi,T)\le r,
\xi^i(x,0)=\td a^i(x),\xi^i_t(x,0)=\td b^i(x)\},\\
&\mathcal S_\te(T,r_1)
:=\{\te\in\Cls10\mid
M_{1,0}(\te,T)\le r_1\}.
\end{aligned}
\end{equation}

We define $T$ and $r$ as follows.
First, for $T=1$, we choose $r>2m_0(\td a,(\td a)',\td b)$ so that $\mathcal S_\gm(1,r)$, $\mathcal S_\xi(1,r)$ and $\mathcal S_\eta(1,r)$ are not empty set.
Next, we choose $T=T_1\le1$ so that the following conditions are satisfied on the time interval $[0,T_1]$.

\begin{enumerate}
\item
The point $\gm(x,t)$ is in the given tubular neighborhood for each $x$.
\item
$B(\gm(x,t),\xi(x,t))\ge B_0>0$.
\end{enumerate}

If $T$ in the expression $\mathcal S_\gm(T,r)$ is sufficiently small, then $\gm(x,t)$ is close to $a(x)$ with respect to $C^0(x)$-topology.
Hence the condition (1) holds.
If $T$ in the expression $\mathcal S_\xi(T,r)$ is sufficiently small, then $\gm(x,t)$ and $\xi(x,t)$ are close to $a(x)$ and $\td a(x)$ respectively with respect to $C^0(x)$-topology, by Lemma \ref{lemma:B(gm,xi)-estimate}.
Hence the condition (2) holds.
Therefore, condition (1) and (2) hold when $T_1$ is sufficiently small.

Let
$
\mathcal S(T)
:=\mathcal S_\gm(T,r)\times\mathcal S_\xi(T,r)\times\mathcal S_\eta(T,r)
$
and $(\gm,\xi,\eta)\in\mathcal S(T_1)$, and solve $\rm(O_\te)$.
Then, the solution $\te$ is, by Proposition \ref{proposition:O-te-sol-and-diff}, estimated as $M_{1,0}(\te,T)\le r_1$ where $r_1$ is a positive constant depending only on $T_1$, $r$ and the given tubular neighborhood.
We define $\mathcal S_\te(T_1,r_1)$ using this $r_1$.

In below, we will rechoose smaller $T$.
By Proposition \ref{proposition:wave-eq-C1-solution},
there is a positive constant $T_2\le T_1$ depending only on $m_0(\td a,(\td a)',\td b)$, $r$, $r_1$ and $T_1$,
such that the solution to equation $\rm(W_\xi)$ satisfies $M_1(\xi,T_2)\le r$
when
$\gm\in\mathcal S_\gm(T_2,r)$,
$\eta\in\mathcal S_\eta(T_2,r)$
and $\te\in\mathcal S_\te(T_2,r_1)$.
Similarly, by proposition \ref{proposition:O-eta-existence},
there is a positive constant $T_3\le T_1$ depending only on $r$, $r_1$ and $T_1$,
such that the solution to equation $\rm(W_\eta)$ satisfies $M_1(\eta,T_3)\le r$
when
$\gm\in\mathcal S_\gm(T_3,r)$,
$\xi\in\mathcal S_\xi(T_3,r)$
and $\te\in\mathcal S_\te(T_3,r_1)$.

Let $T_4:=\min\{T_2,T_3\}$.
For $(\gm,\xi,\eta)\in\mathcal S(T_4)$,
let $\te\in\mathcal S_\te$ be the solution to $\rm(O_\te)$,
and, using them as known functions, we solve $\rm(O_\gm)$, $\rm(W_\xi)$ and $\rm(O_\eta)$,
and let $\td\gm$, $\td\xi$ and $\td\eta$ be solutions.
Then, by the definition of $T_4$, we have $(\td\gm,\td\xi,\td\eta)\in\mathcal S(T_4)$.
It means that the correspondence $\Lm$ is a mapping from $\mathcal S(T_4)$ to $\mathcal S(T_4)$.

Moreover, we can estimate difference of solutions for different data as follows.
For $i=1,2$ and given $(\gm_i,\xi_i,\eta_i)\in\mathcal S(T_4)$,
let $\te_i$ and $(\td\gm_i,\td\xi_i,\td\eta_i)\in\mathcal S(T_4)$ be as above.
Let $\dl*=*_2-*_1$ be the difference.
There is a positive constant $K$ depending only on $T_4$, $r$ and $r_1$ such that
if $T\le T_4$ and
$M_{0,1}(\dl\gm,T)$, $M_0(\dl\Gm,T)$, $M_0(\dl R,T)$,
$M_1(\dl\xi,T)$, $M_{0,1}(\dl\eta,T)$,
$M_{1,0}(\dl\te,T)\le\ep$,
then the followings hold.

\begin{enumerate}
\item
$M_{0,1}(\dl\td\gm,T)\le KT\ep$\quad
(Proposition \ref{proposition:O-gm-difference}),
\item
$M_1(\dl\td\xi,T)\le KT\ep$\quad
(Lemma \ref{lemma:wave-eq-C1-solution-difference}),
\item
$M_0(\dl\td\eta,T)\le KT\ep$ and $M_0(\dl\td\eta_t,T)\le K\ep$\quad
(Proposition \ref{proposition:O-eta-difference}),
\item
$M_{1,0}(\dl\td\te,T)\le K\ep$,\quad
and does not depend on $M_0(\dl\eta_t,T)$
(Proposition \ref{proposition:O-te-sol-and-diff}).
\end{enumerate}

For $(\gm,\xi,\eta)\in\mathcal S(T_4)$,
let $(\td\gm,\td\xi,\td\eta)=\Lm(\gm,\xi,\eta)$,
and let $\td\te$ be the solution to $\rm(O_\te)$ using $(\td\gm,\td\xi,\td\eta)$,
and let $\td{\td\eta}$ be the solution to $\rm(O_\eta)$ using $(\td\gm,\td\xi,\td\te)$.
We define a mapping $\td\Lm:\mathcal S(T_4)\to\mathcal S(T_4)$ by
$\td\Lm(\gm,\xi,\eta)=(\td\gm,\td\xi,\td{\td\eta})$.
Since if $M_{0,1}(\dl\gm,T)$, $M_1(\dl\xi,T)$, $M_{0,1}(\dl\eta,T)\le\ep$
then
$M_{0,1}(\dl\td\gm,T)$, $M_1(\dl\td\xi,T)$, $M_0(\dl\td\eta,T)\le KT\ep$,
we have $M_{1,0}(\dl\td\te,T)\le K^2T\ep$.
Therefore, $M_{0,1}(\dl\td{\td\eta},T)\le K\max\{K,K^2\}T\ep$.
We choose a positive constant $T_5\le T_4$ so that $K\max\{1,K^2\}T_5<1$.
The mapping $\td\Lm:\mathcal S(T_5)\to\mathcal S(T_5)$ is a contraction mapping with respect to the norm
$M_{0,1}(\gm,T_5)+M_1(\xi,T_5)+M_{0,1}(\eta,T_5)$,
hence has a unique fixed point.

Let $(\gm,\xi,\eta)$ be the fixed point.
Then $(\td\gm,\td\xi,\td{\td\eta})=(\gm,\xi,\eta)$,
and
\begin{equation}
\begin{aligned}
&-D_x(D_x\te+\Psi)+\te^\perp=\Phi,
&&-D_x(D_x\td\te+\td\Psi)+\td\te^\perp=\td\Phi,\\
&D_t\td\eta=D_x\te+\td\Psi+D_x\xi,
&&D_t\td{\td\eta}=D_x\td\te+\td{\td\Psi}+D_x\td\xi.
\end{aligned}
\end{equation}
Here, $\td\Psi$ and $\td\Phi$ are $\Psi$ and $\Phi$ substituted $\td\eta$ respectively, and $\td{\td\Psi}$ is $\Psi$ substituted $\td{\td\eta}$.
Therefore, it holds that
\begin{equation}
\begin{aligned}
&-D_x(D_x(\td\te-\te)+(\td\Psi-\Psi))+(\td\te-\te)^\perp=(\td\Phi-\Phi),\\
&D_t\td\eta-D_t\eta=-D_x(\td\te-\te)+(\td\Psi-\Psi).
\end{aligned}
\end{equation}
Hence, by Lemma \ref{lemma:O-te-base-C1},
we see $M_0(D_x(\td\te-\te),T)\le K_6M_0(\td\eta-\eta,T)$,
and by Proposition \ref{proposition:O-eta-difference},
it hols that $M_0(\td\eta-\eta,T)\le K_6TM_0(\td\eta-\eta)$.
Here, $K_6$ is a positive constant depending only on $T_5$, $r$, $r_1$ and $K$.
Therefore, if we replace $T_5$ by a positive constant $T_6\le T_5$ such that $K_6T_6<1$, then we have $\td\eta=\eta$ and $\td\te=\te$.

We proved that the fixed point of the mapping $\td\Lm:\mathcal S(T_6)\to\mathcal S(T_6)$ is a fixed point of the mapping $\Lm:\mathcal S(T_6)\to\mathcal S(T_6)$.
That is, the fixed point is a solution to the coupled system \eqref{eq:base-coupled}.
We summarize the above as follows.
Note that the equality $\xi=\gm_x$ is not yet shown at this stage.

\begin{proposition}
\label{proposition:short-time-C2-existence}
We consider the coupled system \eqref{eq:base-coupled} with initial data
$\gm(x,0)=a(x)\in C^1(x)$,
$\eta(x,0)=b(x)\in C^0(x)$,
$\xi(x,0)=\td a(x)\in C^1(x)$ and $\xi_t(x,0)=\td b(x)\in C^0(x)$.
Assume that they satisfy the length element preserving condition: $\nr{\td a}^2=1$, $g(\td a,\td b)=0$ and the non-geodesic condition: $B(a,\td a)>0$, and that the curve $a(x)$ is regular.
Then, there exists a unique short time solution.
Moreover, the solution satisfies
$\gm\in\Cls02$,
$\xi\in C^1$,
$\eta\in\Cls01$,
and $\te$, $\te_x$, $\pd_x(D_x\te+\Psi)\in C^0$.
\end{proposition}

Next, we study the regularity of solutions.
We use a priori estimate, which is justified by the contraction mapping theorem.
That is, $M_1(D_x\xi)$ and others are bounded when we take the convergent sequence by the contraction mapping.

To show the regularity, we need compatibility conditions for initial data.
Let $\pd/\pd x^i=h^j{}_ie_j$ be the transformation rule of $e_i$ and $\pd/\pd x^i$.
The equality $\xi=\gm_x$ is expressed as
$\td a^ie_i=(a^i)'\pd/\pd x^i=(a^i)'h^j{}_ie_j$.
The equality $D_t\xi=D_x\eta$ is written as
$\td b+\Gm(b,\td a)=b'+\Gm(\td a,b)$,
because
$D_t\xi=\xi_t+\Gm(\eta,\xi)=\td b+\Gm(b,\td a)$
and
$D_x\eta=\eta_x+\Gm(\xi,\eta)=b'+\Gm(\td a,b)$.

\begin{proposition}
\label{proposition:short-time-C3-exsitence}
If the initial data of the short time solution satisfies the length element preserving condition and the compatibility condition $\td a^i=h^i{}_j(a^j)'$ and $\td b+\Gm(b,\td a)=b'+\Gm(\td a,b)$,
and if $a(x)\in C^3(x)$ and $b(x)\in C^2(x)$,
then the short-time solution satisfies $\gm_x=\xi$,
and we have $\gm\in C^3$.
Moreover, it is a solution to the equation of motion \eqref{eq:base-single}.
\end{proposition}

\begin{proof}
Under notations in Lemma \ref{lemma:I(a,b,f,h)-C1-formula},
for $u=I(\td a,\td b,f,h)$, let
\begin{equation}
\begin{aligned}
\wh u^\pm(x,t)&:=D_x\xi(x\mp t,t)\pm\xi_t(x\mp t,t)
=u^\pm(x\mp t,t)+h(x\mp t,t)\\
&=\td a'(x)\pm \td b(x)\pm\int_0^tf(x\mp\tau,\tau)\,d\tau
+h(x,0)+\int_0^th_t(x\mp\tau,\tau)\,d\tau.
\end{aligned}
\end{equation}
By the substitution $x\mp\tau=z$, we have
\begin{equation}
\begin{aligned}
\pd_x&\int_0^tf(x\mp\tau,\tau)\,d\tau
=\pd_x\int_x^{x\mp t}f(z,\mp(z-x))(\mp dz)\\
&=\mp\{f(x\mp t,t)-f(x,0)\}-\int_x^{x\mp t}f_t(z,\mp(z-x))\,dz\\
&=\mp\{f(x\mp t,t)-f(x,0)\}\pm\int_0^tf_t(x\mp\tau,\tau)\,d\tau.
\end{aligned}
\end{equation}
Hence,
\begin{equation}
\begin{aligned}
\wh u^\pm_x(x,t)&=\td a''(x)\pm \td b'(x)
-\{f(x\mp t,t)-f(x,0)\}+\int_0^tf_t(x\mp\tau,\tau)\,d\tau\\
&\qquad+h_x(x,0)
\mp\{h_t(x\mp t,t)-h_t(x,0)\}\pm\int_0^th_{tt}(x\mp\tau,\tau)\,d\tau.
\end{aligned}
\end{equation}

In the expression \eqref{eq:F-H-byGmxietate} of $F$ and $H$,
the $\xi_x$ in $F$ appears only in the form $D_x\xi$.
Thus, $f_t$ is a polynomial in bounded functions $\Gm$, $\Gm_t$, $\Gm_{tt}$, $\xi$, $D_x\xi$, $\xi_t$, $\eta$, $\eta_t$, $\te$, $\pd_t(D_x\xi)$, $\xi_{tt}$, $\eta_{tt}$ and $\te_t$,
and is a polynomial of degree 1 in $\pd_t(D_x\xi)$, $\xi_{tt}$, $\eta_{tt}$ and $\te_t$. Also, $h_{tt}$ is a polynomial of degree 1 in $\xi_{tt}$.
Therefore,
\begin{equation}
\begin{aligned}
m_0(f_t)&\le K_1\{1+m_0(\pd_t(D_x\xi))+m_0(\xi_{tt})+m_0(\eta_{tt})+m_0(\te_t)\}\\
&\le2K_1\{1+m_0(\wh u^+_x)+m_0(\wh u^-_x)+m_0(\eta_{tt})+m_0(\te_t)\},\\
m_0(h_{tt})&\le K_1\{1+m_0(\xi_{tt})\}
\le2K_1\{1+m_0(\wh u^+_x)+m_0(\wh u^-_x)\}.
\end{aligned}
\end{equation}

From this inequality and the equality $\td a''(x)+h_x(x,0)=\pd_xD_x\xi(x,0)$, we see
\begin{equation}
\begin{aligned}
m_0(\wh u^\pm_x)(t)
&\le K_2\{1+\int_0^tm_0(f_t)(\tau)\,d\tau+\int_0^tm_0(h_{tt})\,d\tau\}\\
&\le K_3\{1+\int_0^t(m_0(\wh u^+_x)(\tau)+m_0(\wh u^-_x)(\tau)
+m_0(\eta_{tt})(\tau)+m_0(\te_t)(\tau))\,d\tau\}.
\end{aligned}
\end{equation}

And, from the expression \eqref{eq:O-eta-local-expression} of equation $\rm(O_\eta)$,
\begin{equation}
m_0(\eta_{tt})\le K_4\{1+m_0(\xi_{tt})+m_0(\te_{xt})+m_0(\pd_tD_x\xi)\}.
\end{equation}
Moreover, from equation $\rm(O_\te)$,
\begin{equation}
-D_x(D_x\te_t+\Gm_t(\xi,\te)+\Gm(\xi_t,\te)+\Psi_t)+\te^\perp_t-\Phi_t
\end{equation}
is a bounded function, so
\begin{equation}
\begin{aligned}
m_0(\te_t),m_0(\te_{tx})&\le K_5\{1+m_0(\Psi_t)+m_0(\Phi_t)\}
\le K_6\{1+m_0(\pd_tD_x\xi)+m_0(\xi_{tt})\}.
\end{aligned}
\end{equation}

Combining the above, we have
\begin{equation}
m_0(\wh u^+_x)(t)+m_0(\wh u^-_x)(t)
\le K_7\{1+\int_0^tm_0(\wh u^+_x)(\tau)+m_0(\wh u^-_x)(\tau)\,d\tau\}.
\end{equation}
Hence, $m_0(\wh u^+_x)(t)+m_0(\wh u^-_x)(t)$ is bounded by an exponential function in $t$.
Therefore, functions $\pd_tD_x\xi$, $\xi_{tt}$, $\eta_{tt}$, $\te_t$ and $\te_{xt}$ are bounded.

Next, we study the continuity of $\wh u^\pm_x$.
Since functions $f$ and $h_t$ are continuous and functions $f_t$ and $h_{tt}$ are bounded, $\wh u^\pm_x$ is continuous with respect to $t$.
For the continuity with respect to $x$,
we see that functions $\td a''$, $\td b'$, $f$, $h_t$, $h_x(x,0)$ and $\te_{tx}$ are continuous with respect to $x$,
and by a similar calculation to the above, we have
\begin{equation}
\begin{aligned}
\nr{\dl\wh u^\pm_x(x,t)}
&\le K_8\{O(\ep)+\int_0^t\nr{\dl\wh u^+_x(x,\tau)}+\nr{\dl\wh u^-_x(x,\tau)}
+\nr{\dl\eta_{tt}(x,\tau)}\,d\tau\},\\
\nr{\dl\eta_{tt}(x,t)}
&\le K_8\{O(\ep)+\nr{\dl\wh u^+_x(x,t)}+\nr{\dl\wh u^-_x(x,t)}\}.
\end{aligned}
\end{equation}
Here, $\dl\wh u^\pm_x(x,t)=\wh u^\pm_x(x+\ep,t)-\wh u^\pm_x(x,t)$ and
$\dl\eta_{tt}(x,t)=\eta_{tt}(x+\ep,t)-\eta_{tt}(x,t)$.
Therefore, $\nr{\dl\wh u^\pm_x(x,t)}$, $\nr{\dl\eta_{tt}(x,t)}\to0$ when $\ep\to0$.

Moreover, since $\wh u^\pm_t=\pm f(x\mp t,t)+h_t(x\mp t,t)$ and it is continuous, $\wh u^\pm$ is of class $C^1$, and so $D_x\xi$, $\xi_t\in C^1$, $\eta_{tt}\in C^0$ and $\te_t$, $\te_{tx}\in C^0$.
We don't know whether $\xi_x$ is of class $C^1$ at this stage,
but we know that $D_x^2\xi$ and $D_t^2\xi$ are continuous.
We also know that $D_xD_t\eta=D_x(D_x\te+\Psi)+D_x^2\xi$ is continuous, and so $\eta_{tx}$, $\eta_x$, $\gm_{tx}$ and $\gm_x$ are continuous.
Therefore, $\gm_{xt}=\gm_{tx}$ and $\eta_{xt}=\eta_{tx}$, and $D_t\gm_x=\Dt\gm_x=\Dx\gm_t=\Dx\eta=D_x\eta+\Gm(\gm_x-\xi,\eta)$.
Thus, the compatibility condition of initial data implies $D_t\xi=D_x\eta=D_t\gm_x$ at $t=0$, and we have $\gm_x=\xi$ by Proposition \ref{proposition:restore-base-single}.

Now we have $\xi_x=D_x\xi-\Gm(\xi,\xi)\in C^1$, and $\gm_x=\xi\in C^2$.
Also, $D_t^2\gm_t=D_t^2\eta$ is continuous, therefore $\gm\in C^3$.
\end{proof}

\section
{Existence of a long-time solution}
\label{section:long-time-existence}

Let $\gm(x,0)\in C^3(x)$ and $\gm_t(x,0)\in C^2(x)$ be initial values.
By Proposition \ref{proposition:short-time-C2-existence} and Proposition \ref{proposition:short-time-C3-exsitence}, the solution is extended as long as
\begin{equation}
M(\gm,\xi,\eta):=m_0(\Dx\xi)+m_0(\Dt\xi)+m_0(\eta)
\end{equation}
is bounded and $B(\gm,\gm_x)>B_0>0$.
We assume $B(\gm,\gm_x)>B_0>0$, and study the behaviour of $M(\gm,\xi,\eta)$.
We call the solution a $C^3$ {\em solution to the equation of motion} \eqref{eq:base-single}.
First, we prove the total energy preserving law.

\begin{proposition}
\label{proposition:olE-preserving}
The total energy $\ol E(\gm):=\Nr{\Dt\gm_x}^2+\Nr{\gm_t}^2+\Nr{\Dx\gm_x}^2$ of the $C^3$ solution to the equation of motion is preserved.
\end{proposition}

\begin{proof}
Each term of the right hand side of equation
\begin{equation}
\frac12\ddt\{\Nr{\Dt\gm_x}^2+\Nr{\gm_t}^2+\Nr{\Dx\gm_x}^2\}
=\Ip{\Dt\xi,\Dt^2\xi}+\Ip{\eta,\Dt\eta}+\Ip{\Dx\xi,\Dt\Dx\xi}
\end{equation}
becomes
\begin{equation}
\begin{aligned}
&\Ip{\Dt\xi,\Dt^2\xi}=\Ip{\Dt\xi,\Dx^2\xi+\te},\\
&\Ip{\eta,\Dt\eta}=\Ip{\eta,\Dx\te+\Psi+\Dx\xi}
=-\Ip{\Dx\eta,\te+\xi}+\Ip{\eta,R(\xi,\Dx\xi)\xi}\\
&\qquad=-\Ip{\Dt\xi,\te}+\Ip{R(\xi,\Dx\xi)\xi,\eta},\\
&\Ip{\Dx\xi,\Dt\Dx\xi}=\Ip{\Dx\xi,R(\eta,\xi)\xi+\Dx\Dt\xi}
=\Ip{R(\eta,\xi)\xi,\Dx\xi}-\Ip{\Dx^2\xi,\Dt\xi}.
\end{aligned}
\end{equation}
The sum of them is $0$.
\end{proof}

From this proposition, we have
\begin{equation}
\begin{aligned}
&\nr{\gm_t(x,t)}\le m_0(\eta)(t)
\le\Nr{\eta}+\Nr{\Dx\eta}
=\Nr{\eta}+\Nr{\Dt\xi}
\le2\ol E(\gm),\\
&\operatorname{distance}(\gm(x,t),\gm(0,0))
\le1+\int_0^t\nr{\gm_t(0,\tau)}\,d\tau
\le1+2\ol E(\gm)t,
\end{aligned}
\end{equation}
and get the following proposition.

\begin{proposition}
The $C^3$ solution is of finite velocity, and in particular the curvature tensor and its covariant derivatives are bounded along the solution on finite time interval $0\le t<T$.
\end{proposition}

Next, we estimate $M(\gm,\xi,\eta)=m_0(\xi_x)+m_0(\xi_t)+m_0(\eta)$.
We cannot directly use the estimation for short-time existence, because the coordinate system changes.
Since $\eta$ is bounded as above, we only have to show boundedness of $\xi_x$ and $\xi_t$.

\begin{lemma}
\label{lemma:m0(xix)m0(xit)-bounded}
We assume $B(\gm,\gm_x)\ge B_0>0$ on a finite time interval $0\le t<T$.
Then, $m_0(\xi_x)+m_0(\xi_t)$ is bounded on the time interval.
\end{lemma}

\begin{proof}
Let $\xi^\pm(x,t):=\Dx\xi(x,t)\pm\Dt\xi(x,t)$
and $\wh\xi^\pm(x,t):=\xi^\pm(x\mp t,t)$.
We have

\begin{equation}
\begin{aligned}
\Dt\wh\xi^\pm&=\mp\Dx\xi^\pm+\Dt\xi^\pm
=\mp\Dx^2\xi-\Dx\Dt\xi+\Dt\Dx\xi\pm\Dt^2\xi\\
&=\pm\{(\nr{\Dx\xi}^2-\nr{\Dt\xi}^2)\xi+\te^\perp\}+R(\eta,\xi)\xi.
\end{aligned}
\end{equation}
We used \eqref{eq:xi-wave-rep} in the last equality.
Noting that $\wh\xi^\pm$ is orthogonal to $\xi$, we have
\begin{equation}
\pd_t\nr{\wh\xi^\pm}^2=2g(\wh\xi^\pm,\Dt\wh\xi^\pm)
=\pm2g(\wh\xi^\pm,\te^\perp).
\end{equation}

Since $\ol E(\gm)$ is constant by Proposition \ref{proposition:olE-preserving},
$\Nr{D_x\xi}$ is bounded.
Therefore, by Lemma \ref{lemma:O-te-base},
\begin{equation}
\begin{aligned}
m_0(\te)&\le C(\Nr{\Psi}+\Nr{\Phi})\\
&\le C((m_0(\Dx\xi)+m_0(\Dt\xi))
+(\Nr{\Dt\xi}m_0(\Dt\xi)+\Nr{\Dx\xi}m_0(\Dx\xi)+1))\\
&\le C(1+\ol E(\gm))(1+m_0(\Dx\xi)+m_0(\Dt\xi)).
\end{aligned}
\end{equation}
Therefore
\begin{equation}
\ddt m_0(\wh\xi^\pm)^2\le C(1+m_0(\wh\xi^\pm)^2),
\end{equation}
and $m_0(\wh\xi^\pm)$, $m_0(\Dx\xi)$ and $m_0(\Dt\xi)$ increase at most exponential order, and so are bounded.
\end{proof}

From the above, we can prove the long-time existence of solutions.

\begin{theorem}
\label{theorem:long-time-C3-existence}
Assume that the Riemannian manifold $(M,g)$ is complete.
The $C^3$ solution $\gm$ can be extended to an infinite time solution, provided that $\inf_{0\le t<T}B(\gm)>0$ for any finite time $T$.
\end{theorem}

\begin{proof}
Let $T<\infty$ be the supremum of the existence time of the solution.
Then $B(\gm)\ge B_0>0$ on $0\le t<T$ by the assumption,
and $M(\gm,\xi,\eta)$ is bounded on $0\le t<T$ by Lemma \ref{lemma:m0(xix)m0(xit)-bounded}, hence the solution extends over $T$.
\end{proof}

To make sure, we study $C^\infty$ solutions.
For that, we estimate derivatives of $\te$.

\begin{lemma}
\label{lemma:te-Cn-estimate}
Let $n\ge2$.
If $\gm\in C^n$, then $\te$, $\te_x\in C^{n-2}$, and $C^{n-2}$-norms of $\te$ and $\te_x$ are bounded from above by a constant depending on $C^n$-norm of $\gm$.
\end{lemma}

\begin{proof}
By Proposition \ref{proposition:O-te-sol-and-diff}, the statement is true for $n=2$.
We assume that the statement is true for less than $n$ and prove that the statement is true for $n$.
If $\gm\in C^n$, then $\Psi$, $\Phi\in C^{n-2}$.

We denote by $O(p,q)$ a polynomial in functions $\pd_x^i\pd_t^j\gm$ $(i+j\le n)$ and $\pd_x^i\pd_t^j\te$ $(i\le p,j\le q)$.
Since $\Dx\Dt u-\Dt\Dx u=R(\xi,\eta)u$, changing order of $\Dx$ and $\Dt$ in a operation $p$ times $\Dx$ and $q$ times $\Dt$ leads difference $O(p-1,q-1)$.

Let $p+q=n-2$.
Then, by the assumption of induction, $O(p-1,q)$ and $O(p+1,q-1)$ are continuous and bounded.
And,
\begin{equation}
\begin{aligned}
&\Dx^p\Dt^q(\te^\perp)-(\Dx^p\Dt^q\te)^\perp
=\Dx^p\Dt^q(\te-g(\te,\xi)\xi)-(\Dx^p\Dt^q\te-g(\Dx^p\Dt^q\te,\xi)\xi)\\
&\qquad=-\Dx^p\Dt^q(g(\te,\xi)\xi)+g(\Dx^p\Dt^q\te,\xi)\xi
=O(p-1,q)+O(p,q-1),\\
&\Dx^p\Dt^q(\Dx^2\te)-\Dx^2(\Dx^p\Dt^q\te)
=O(p+1,q-1),\\
&\Dx^p\Dt^q\Dx\Psi-\Dx^{p+1}\Dt^q\Psi
=O(0,0),
\end{aligned}
\end{equation}
provided that the term $O(p-1,q)$ and $O(p,q-1)$ do not appear when $p=0$ and $q=0$, respectively.
Hence,
\begin{equation}
\begin{aligned}
&\Dx(\Dx\Dx^p\Dt^q\te+\Dx^p\Dt^q\Psi)+(\Dx^p\Dt^q\te)^\perp
=\Dx^p\Dt^q\Phi+O(p-1,q)+O(p+1,q-1).
\end{aligned}
\end{equation}
Therefore, by Corollary \ref{corollary:O-te-t-dependence},
$\Dx^p\Dt^q\te$ and $\Dx^{p+1}\Dt^q\te$ are continuous and bounded.
That is, $\te$ and $\te_x$ are of class $C^{n-2}$.
\end{proof}

\begin{theorem}
\label{theorem:long-time-Cinfty-existence}
If the initial data satisfies $\gm(x,0)\in C^\infty(x)$ and $\gm_t(x,0)\in C^\infty(x)$,
then the $C^3$ solution is of class $C^\infty$.
\end{theorem}

\begin{proof}
Since we know that $\Gm_t$ is continuous, we consider a hyperbolic equation for $u=\xi$ using
\begin{equation}
\begin{aligned}
&f(x,t)
=-\pd_t(\Gm(\eta,\xi))+\pd_x(\Gm(\xi,\xi))
-\Gm(\eta,\Dt\xi)+\Gm(\xi,\Dx\xi)\\
&\qquad\qquad\qquad+(\nr{\Dx\xi}^2-\nr{\Dt\xi}^2)\xi
+\te^\perp,\\
&h(x,t)=0
\end{aligned}
\end{equation}
instead of equation \eqref{eq:F-H-byGmxietate}.
Let $n\ge 3$, and consider on the finite time interval $0\le t<T$.
We will show boundedness of the $C^{n+1}$-norm of $\gm$, assuming boundedness of the $C^n$-norm of $\gm$.

Assume that the $C^n$-norm of $\gm$ is bounded.
From the expression of the equation and Lemma \ref{lemma:te-Cn-estimate}, $f$ is of class $C^{n-2}$.
Let $p+q=n-2$ and consider $\pd_x^p\pd_t^q\xi$.
Since $\pd_x^p\pd_t^{q+2}\xi-\pd_x^{p+2}\pd_t^q\xi=\pd_x^p\pd_t^qf\in C^0$,
it holds that $\pd_x^p\pd_t^{q+2}\xi-\pd_x^{p+q+2}\xi\in C^0$
or
$\pd_x^p\pd_t^{q+2}\xi-\pd_x^{p+q+1}\pd_t\xi\in C^0$.
Moreover, we have
$\pd_x^{p+q+2}\xi(x,0)$, $\pd_x^{p+q+1}\pd_t\xi(x,0)\in C^\infty(x)$ from the assumption,
so we get
$\pd_x^p\pd_t^q\xi(x,0)\in C^2(x)$ and
$\pd_x^p\pd_t^q\pd_t\xi(x,0)\in C^1(x)$.

Therefore, $u=\pd_x^p\pd_t^q\xi$ is a $C^1$ solution to the integral wave equation
\begin{equation}
2u
=a(x+t)+a(x-t)+\int_{x-t}^{x+t}b(y)\,dy
+\int_0^t\int_{x-(t-\tau)}^{x+(t-\tau)}\pd_x^p\pd_t^qf(y,\tau)\,dyd\tau
\end{equation}
with initial values
$a(x)=\pd_x^p\pd_t^q\xi(x,0)\in C^2(x)$
and $b(x)=\pd_x^p\pd_t^q\pd_t\xi(x,0)\allowbreak\in C^1(x)$.

And, we can calculate as follows, in the same way as the wave equation for $\xi$.
Let $u^\pm(x,t)=u_x(x,t)\pm u_t(x,t)$ and $\wh u^\pm(x,t)=u^\pm(x\mp t,t)$.
By \eqref{eq:vpm}, we have
\begin{equation}
\wh u(x,t)
=a'(x)\pm b(x)\pm\int_0^t\pd_x^p\pd_t^qf(x\mp\tau,\tau)\,d\tau.
\end{equation}
From this, $\wh u_t(x,t)=\pd_x^p\pd_t^qf(x\mp t,t)\in C^0$.
Besides,
\begin{equation}
\wh u_x(x,t)
=a''(x)\pm b'(x)\pm\int_0^t(\pd_x^p\pd_t^qf)_x(x\mp\tau,\tau)\,d\tau,
\end{equation}
and there is a positive constant $K$ depending only on $C^n$-norm of $\gm$ such that
\begin{equation}
m_0((\pd_x^p\pd_t^qf)_x)
\le K\{1+m_0(\pd_x^2u)+m_0(\pd_x\pd_tu)+m_0(\pd_t^2u)
+m_0(\pd_x^p\pd_t^q\te_x)\}.
\end{equation}
Hence $m_0(\wh u^\pm_x)$ is bounded and $u$ is $C^2$-bounded.
And, we can show $u\in C^2$ in a similar way.
Therefore, $\xi\in C^n$.

Here, $\pd_t^{n-1}\Psi$ and $\pd_t^{n-1}\Phi$ contain $\pd_t^n\xi$ or less for $\xi$,
but contain only $\pd_t^{n-1}\eta$ or less for $\eta$,
hence $\pd_t^{n-1}\Psi$, $\pd_t^{n-1}\Phi\in C^0$.
Then, by a similar manner as Lemma \ref{lemma:te-Cn-estimate},
we see $\pd_t^{n-1}\te_x\in C^0$.
Therefore, $\Dt^n\eta=\Dt^{n-1}(\Dx\te+\Psi)+\Dt^{n-1}\Dx\xi\in C^0$.
Now we conclude that derivatives of $\gm$ of order $n+1$ are continuous and bounded,
that is, the induction holds.
\end{proof}


\end{document}